\documentclass[12pt]{amsart}
\usepackage{amsmath}
\usepackage{amsfonts}
\usepackage{amssymb}
\usepackage{bbding}
\usepackage{stmaryrd}
\usepackage{txfonts}
\usepackage{graphicx}
\usepackage{epsfig}
\usepackage{xypic}
\usepackage{tikz}
\usepackage{longtable}
\usepackage[all]{xy}
\usepackage{pgflibraryarrows}
\usepackage{pgflibrarysnakes}
\usepackage[shortlabels]{enumitem}
\usepackage{ifpdf}
\ifpdf
  \usepackage[colorlinks,final,backref=page,hyperindex]{hyperref}
\else
  \usepackage[colorlinks,final,backref=page,hyperindex,hypertex]{hyperref}
\fi
\usepackage{xspace}

\newtheorem{theorem}{Theorem}[section]
\newtheorem{lemma}[theorem]{Lemma}
\newtheorem{coro}[theorem]{Corollary}

\newtheorem{prop}[theorem]{Proposition}
\theoremstyle{definition}
\newtheorem{defn}[theorem]{Definition}

\newtheorem{exam}[theorem]{Example}

\newtheorem{convention}[theorem]{Convention}
\newtheorem{const}[theorem]{Construction}

\newcommand{\nc}{\newcommand}
\newcommand{\delete}[1]{}

\nc{\tred}[1]{\textcolor{red}{#1}}
\nc{\tblue}[1]{\textcolor{blue}{#1}} \nc{\tgreen}[1]{\textcolor{green}{#1}} \nc{\tpurple}[1]{\textcolor{purple}{#1}} \nc{\btred}[1]{\textcolor{red}{\bf #1}} \nc{\btblue}[1]{\textcolor{blue}{\bf #1}} \nc{\btgreen}[1]{\textcolor{green}{\bf #1}} \nc{\btpurple}[1]{\textcolor{purple}{\bf #1}}

\newcommand{\efootnote}[1]{}

\nc{\mlabel}[1]{\label{#1}}  
\nc{\mcite}[2][]{\cite[#1]{#2}}  
\nc{\mref}[1]{\ref{#1}}  
\nc{\mbibitem}[1]{\bibitem{#1}} 

\delete{
\nc{\mlabel}[1]{\label{#1}  
{\hfill \hspace{1cm}{\bf{{\ }\hfill(#1)}}}}
\nc{\mcite}[1]{\cite{#1}}  
\nc{\mref}[1]{\ref{#1}{{\bf{{\ }(#1)}}}}  
\nc{\mbibitem}[1]{\bibitem[\bf #1]{#1}} 
}

\renewcommand\geq{\geqslant}
\renewcommand\leq{\leqslant}

\renewcommand\bar[1]{\overline{#1}}

\nc{\fA}{{\mathfrak A}}
\nc{\myone}{{\bf 1}}
\nc{\bfu}{{\bf u}}
\nc{\bfv}{{\bf v}}

\nc{\nz}{\varepsilon}
\nc{\Id}{\mathrm{Id}}
\nc{\map}[2]{{#2}^{#1}}
\nc{\gp}{B}
\nc{\wcomp}{\large{\VDash}}
\nc{\set}{\mathrm{set}}

\nc{\Irr}{\mathrm{Irr}}
\nc{\vx}{\sigma} \nc{\vy}{\tau} \nc{\dvx}{\sigma^{(1)}} \nc{\dvy}{\tau^{(1)}} \nc{\done}{\vep} \nc{\mcitep}[1]{\mcite{#1}} \nc{\wt}{\mathrm{wt}} \nc{\bre}[1]{|#1|} \nc{\mapmonoid}{\frakM} \nc{\disjoint}{\frakM'}
\nc{\ncpoly}[1]{\langle #1\rangle}  
\nc{\mapm}[1]{\lfloor\!|{#1}|\!\rfloor}
\nc{\diff}[1]{{}^\NC\{ #1 \}} \nc{\disj}[1]{\{{#1}\}'} \nc{\mdisj}[1]{\frakM'(#1)} \nc{\brho}{\bar{\rho}} \nc{\om}{\bar{\frakm}} \nc{\frakn}{\mathfrak n} \nc{\ddeg}[1]{^{(#1)}} \nc{\opset}{X} \nc{\genset}{{Z}} \nc{\NC}{\mathrm{{NC}}} \nc{\leaf}{\mathrm{leaf}} \nc{\twig}{\mathrm{twig}} \nc{\fe}{\mathrm{fl}} \nc{\munderline}[1]{#1} \nc{\bo}{o} \nc{\dep}{\mathrm{depth}} \nc{\ofe}{\mathrm{ofl}} \nc{\dfe}{\mathrm{dfe}} \nc{\fex}{\mathrm{fex}} \nc{\dl}{\mathrm{dlex}} \nc{\db}{\mathrm{db}} \nc{\lex}{\mathrm{lex}} \nc{\clex}{\mathrm{clex}} \nc{\dgp}{\mathrm{dgp}} \nc{\dgx}{\mathrm{dgx}} \nc{\br}{\mathrm{br}} \nc{\obd}{\mathrm{odb}} \nc{\ob}{\mathrm{ob}}
\nc{\pie}{\mathrm{PIE}}
\nc{\rbo}{\mathrm{RBO}}
\nc{\supp}{\mathcal{S}}
\nc{\nul}{\mathcal{Z}}

\nc{\bin}[2]{ (_{\stackrel{\scs{#1}}{\scs{#2}}})}  
\nc{\binc}[2]{ \left (\!\! \begin{array}{c} \scs{#1}\\
    \scs{#2} \end{array}\!\! \right )}  
\nc{\bincc}[2]{  \left ( {\scs{#1} \atop
    \vspace{-1cm}\scs{#2}} \right )}  
\nc{\bs}{\bar{S}} \nc{\cosum}{\sqsubset} \nc{\la}{\longrightarrow} \nc{\rar}{\rightarrow} \nc{\dar}{\downarrow} \nc{\dprod}{**} \nc{\dap}[1]{\downarrow \rlap{$\scriptstyle{#1}$}} \nc{\md}[1]{\bar{#1}} \nc{\uap}[1]{\uparrow \rlap{$\scriptstyle{#1}$}} \nc{\defeq}{\stackrel{\rm def}{=}} \nc{\disp}[1]{\displaystyle{#1}} \nc{\dotcup}{\ \displaystyle{\bigcup^\bullet}\ } \nc{\gzeta}{\bar{\zeta}} \nc{\hcm}{\ \hat{,}\ } \nc{\hts}{\hat{\otimes}} \nc{\barot}{{\otimes}} \nc{\free}[1]{\bar{#1}} \nc{\uni}[1]{\tilde{#1}} \nc{\hcirc}{\hat{\circ}} \nc{\leng}{\ell} \nc{\lleft}{[} \nc{\lright}{]} \nc{\lc}{\lfloor} \nc{\rc}{\rfloor}
\nc{\lb}{[} 
\nc{\rb}{]} 
\nc{\curlyl}{\left \{ \begin{array}{c} {} \\ {} \end{array}
    \right.  \!\!\!\!\!\!\!}
\nc{\curlyr}{ \!\!\!\!\!\!\!
    \left. \begin{array}{c} {} \\ {} \end{array}
    \right \} }
\nc{\longmid}{\left | \begin{array}{c} {} \\ {} \end{array}
    \right. \!\!\!\!\!\!\!}
\nc{\onetree}{\bullet} \nc{\ora}[1]{\stackrel{#1}{\rar}}
\nc{\ola}[1]{\stackrel{#1}{\la}}
\nc{\ot}{\otimes} \nc{\mot}{{{\boxtimes\,}}} \nc{\otm}{\overline{\boxtimes}} \nc{\sprod}{\bullet} \nc{\scs}[1]{\scriptstyle{#1}} \nc{\mrm}[1]{{\rm #1}} \nc{\msum}{\sum\limits}
\nc{\margin}[1]{\marginpar{\rm #1}}   
\nc{\dirlim}{\displaystyle{\lim_{\longrightarrow}}\,} \nc{\invlim}{\displaystyle{\lim_{\longleftarrow}}\,} \nc{\mvp}{\vspace{0.3cm}} \nc{\tk}{^{(k)}} \nc{\tp}{^\prime} \nc{\ttp}{^{\prime\prime}} \nc{\svp}{\vspace{2cm}} \nc{\vp}{\vspace{8cm}} \nc{\proofbegin}{\noindent{\bf Proof: }}
\nc{\proofend}{$\blacksquare$ \vspace{0.3cm}}
\nc{\modg}[1]{\!<\!\!{#1}\!\!>}
\nc{\intg}[1]{F_C(#1)} \nc{\lmodg}{\!<\!\!} \nc{\rmodg}{\!\!>\!} \nc{\cpi}{\widehat{\Pi}}
\nc{\sha}{{\mbox{\cyr X}}}  
\nc{\shap}{{\mbox{\cyrs X}}} 
\nc{\shpr}{\diamond}    
\nc{\shp}{\ast} \nc{\shplus}{\shpr^+}
\nc{\shprc}{\shpr_c}    
\nc{\msh}{\ast} \nc{\zprod}{m_0} \nc{\oprod}{m_1} \nc{\vep}{\varepsilon} \nc{\labs}{\mid\!} \nc{\rabs}{\!\mid}
\nc{\astarrow}{\overset{\raisebox{-3pt}{$\ast$}}{\rightarrow}}

\nc{\LWC}{\mrm{LWC}}
\nc{\sym}{\mrm{Sym}}
\nc{\qsym}{\mrm{QSym}}
\nc{\syms}{symmetric functions\xspace}
\nc{\eqsym}{LWC quasi-symmetric function\xspace}
\nc{\eqsyms}{LWC quasi-symmetric functions\xspace}
\nc{\Eqsyms}{LWC quasi-symmetric functions\xspace}
\nc{\Esyms}{LWC symmetric functions\xspace}
\nc{\EQSYM}{\mrm{LWCQSym}}
\nc{\sgqsym}{quasi-symmetric function with semigroup exponents\xspace}
\nc{\sgqsyms}{quasi-symmetric functions with semigroup exponents\xspace}
\nc{\Sgqsyms}{Quasi-symmetric functions with semigroup exponents\xspace}
\nc{\lwc}{LWC\xspace}
\nc{\SGQSYM}{\mrm{SGQSYM}}
\nc{\emzv}{left weak composition multiple zeta value\xspace}
\nc{\emzvs}{left weak composition multiple zeta values\xspace}
\nc{\lwcmzv}{LWCMZV\xspace}
\nc{\lwcmzvs}{LWCMZVs\xspace}
\nc{\sgfps}{formal power series with semigroup exponent\xspace}
\nc{\nsymg}{\mathrm{NSym}_\gp}

\nc{\parr}{\rm Par}
\nc{\wpar}{\rm WPar}

\nc{\dth}{d} \nc{\mmbox}[1]{\mbox{\ #1\ }} \nc{\fp}{\mrm{FP}} \nc{\rchar}{\mrm{char}} \nc{\Fil}{\mrm{Fil}} \nc{\Mor}{Mor\xspace} \nc{\gmzvs}{gMZV\xspace} \nc{\gmzv}{gMZV\xspace} \nc{\mzv}{MZV\xspace} \nc{\mzvs}{MZVs\xspace} \nc{\Hom}{\mrm{Hom}} \nc{\id}{\mrm{id}} \nc{\im}{\mrm{im}} \nc{\incl}{\mrm{incl}}  \nc{\mchar}{\rm char}

\nc{\Alg}{\mathbf{Alg}} \nc{\Bax}{\mathbf{Bax}} \nc{\bff}{\mathbf f} \nc{\bfk}{{\bf k}} \nc{\bfone}{{\bf 1}} \nc{\bfx}{\mathbf x} \nc{\bfy}{\mathbf y}
\nc{\base}[1]{\bfone^{\otimes ({#1}+1)}} 
\nc{\Cat}{\mathbf{Cat}} \delete{}
\nc{\detail}{\marginpar{\bf More detail}
    \noindent{\bf Need more detail!}
    \svp}
\nc{\Int}{\mathbf{Int}} \nc{\Mon}{\mathbf{Mon}}
\nc{\rbtm}{{shuffle }}
\nc{\rbto}{{Rota--Baxter }} \nc{\remarks}{\noindent{\bf Remarks: }} \nc{\Rings}{\mathbf{Rings}} \nc{\Sets}{\mathbf{Sets}}
\nc{\balpha}{\mathbf{\alpha}}

\nc{\BA}{{\mathbb A}} \nc{\CC}{{\mathbb C}} \nc{\DD}{{\mathbb D}} \nc{\EE}{{\mathbb E}} \nc{\FF}{{\mathbb F}} \nc{\GG}{{\mathbb G}} \nc{\HH}{{\mathbb H}} \nc{\LL}{{\mathbb L}} \nc{\NN}{{\mathbb N}} \nc{\KK}{{\mathbb K}} \nc{\PP}{{\mathbb P}} \nc{\QQ}{{\mathbb Q}} \nc{\RR}{{\mathbb R}} \nc{\TT}{{\mathbb T}} \nc{\VV}{{\mathbb V}} \nc{\ZZ}{{\mathbb Z}}

\nc{\cala}{{\mathcal A}} \nc{\calc}{{\mathcal C}} \nc{\cald}{{\mathcal D}} \nc{\cale}{{\mathcal E}} \nc{\calf}{{\mathcal F}} \nc{\calg}{{\mathcal G}} \nc{\calh}{{\mathcal H}} \nc{\cali}{{\mathcal I}} \nc{\call}{{\mathcal L}} \nc{\calm}{{\mathcal M}} \nc{\caln}{{\mathcal N}} \nc{\calo}{{\mathcal O}} \nc{\calp}{{\mathcal P}} \nc{\calr}{{\mathcal R}} \nc{\cals}{{\mathcal S}} \nc{\calt}{{\mathcal T}} \nc{\calw}{{\mathcal W}} \nc{\calk}{{\mathcal K}} \nc{\calx}{{\mathcal X}}
\nc{\calz}{{\mathcal Z}}
 \nc{\CA}{\mathcal{A}}

\nc{\fraka}{{\mathfrak a}} \nc{\frakA}{{\mathfrak A}} \nc{\frakb}{{\mathfrak b}} \nc{\frakB}{{\mathfrak B}}
\nc{\frakc}{{\mathfrak c}}  \nc{\frakD}{{\mathfrak D}}
\nc{\frakH}{{\mathfrak H}}
\nc{\frakh}{{\mathfrak h}} \nc{\frakM}{{\mathfrak M}}
\nc{\frakO}{{\mathfrak O}}
\nc{\frakE}{{\mathfrak E}}
\nc{\bfrakM}{\overline{\frakM}} \nc{\frakm}{{\mathfrak m}} \nc{\frakP}{{\mathfrak P}} \nc{\frakN}{{\mathfrak N}} \nc{\frakp}{{\mathfrak p}} \nc{\frakS}{{\mathfrak S}}
\nc{\frakk}{{\mathfrak k}}
\nc{\frakx}{{\mathfrak x}}
\nc{\frakl}{{\mathfrak l}} \nc{\ox}{\bar{\frakx}} \nc{\frakX}{{\mathfrak X}} \nc{\fraky}{{\mathfrak y}} \nc\dop{\delta}
\nc{\Reduce}{{\rm Red}}

\font\cyr=wncyr10 \font\cyrs=wncyr7
\nc{\redt}[1]{\textcolor{red}{#1}}
\nc{\li}[1]{\textcolor{red}{\tt Li:#1}}
\nc{\yu}[1]{\textcolor{blue}{\tt Yu:#1}}
\nc{\jz}[1]{\textcolor{green}{\tt JZ:#1}}

\begin{document}
\title{Rota--Baxter algebras and left weak composition quasi-symmetric functions}

\author{Li Guo}
\address{
Department of Mathematics and Computer Science, Rutgers University, Newark, NJ 07102, USA}
\email{liguo@rutgers.edu}

\author{Houyi Yu}
\address{School  of Mathematics and Statistics, Southwest University, Chongqing, 400715, China}
\email{yuhouyi@swu.edu.cn}

\author{Jianqiang Zhao}
\address{ICMAT,  C/Nicol\'as Cabrera, no.~13-15, 28049 Madrid, Spain}
\email{zhaoj@ihes.fr}

\hyphenpenalty=8000
\date{\today}

\begin{abstract}
Motivated by a question of Rota, this paper studies the relationship between Rota--Baxter algebras and symmetric related functions. The starting point is the fact that the space of quasi-symmetric functions is spanned by monomial quasi-symmetric functions which are indexed by compositions. When composition is replaced by left weak composition (LWC), we obtain the concept of LWC monomial quasi-symmetric functions and the resulting space of LWC quasi-symmetric functions. In line with the question of Rota, the latter is shown to be isomorphic to the free commutative nonunitary Rota--Baxter algebra on one generator. The combinatorial interpretation of quasi-symmetric functions by $P$-partitions from compositions is extended to the context of left weak compositions, leading to the concept of LWC fundamental quasi-symmetric functions. The transformation formulas for LWC monomial and LWC fundamental quasi-symmetric functions are obtained, generalizing the corresponding results for quasi-symmetric functions. Extending the close relationship between quasi-symmetric functions and multiple zeta values, weighted multiple zeta values and a $q$-analog of multiple zeta values are investigated and a decomposition formula is established.
\end{abstract}

\subjclass[2010]{05E05,16W99,11M32}

\keywords{Rota-Baxter algebras, symmetric functions, quasi-symmetric functions,  left weak compositions, monomial quasi-symmetric functions, fundamental quasi-symmetric functions, $P$-partitions, lmultiple zeta values, $q$-multiple zeta values}

\maketitle

\tableofcontents

\hyphenpenalty=8000 \setcounter{section}{0}


\allowdisplaybreaks

\section{Introduction}\label{sec:int}

Symmetric functions have played an important role in mathematics for a long time~\cite{Ma,Sta2}. More recently, their generalizations, especially quasi-symmetric functions and noncommutative symmetric functions, have played similar roles.
Quasi-symmetric functions were defined explicitly by Gessel~\cite{Ge} in 1984 with motivation from the $P$-partitions of Stanley~\cite{St1} in 1972. Most of their study was carried out after the middle 1990s. The last decade witnessed a surge in the research on quasi-symmetric functions with broad applications. Conceptually, the central role played by quasi-symmetric functions was demonstrated by the result of Aguiar, Bergeron and Sottile~\cite{ABS} that the Hopf algebra of quasi-symmetric functions is the terminal object in the category of combinatorial Hopf algebras. This Hopf algebra is the graded dual of the Hopf algebra of noncommutative symmetric functions~\cite{NCS}, another important generalization of symmetric functions. Generalizations of quasi-symmetric functions have also been introduced (see for example~\cite{HK,NTT}). Further details of quasi-symmetric functions can be found in the monograph~\cite{LMW} and the references therein.
In this paper we study the close relationship between Rota--Baxter algebras and the algebras of quasi-symmetric functions, and investigate a new class of
multiple zeta values that arise from this relationship.

Rota--Baxter algebra is an abstraction of the algebra of continuous functions equipped with the integration operator, characterized by the integration by parts formula. Shortly after its introduction by G. Baxter~\cite{Ba} in 1960 from his probability study, Rota--Baxter algebra (called Baxter algebra in the early literature) attracted the interest of well-known combinatorists such as Cartier and Rota. Rota gave the first construction of free commutative Rota--Baxter algebras~\cite{Ro1,RS} and thereby showed that the Spitzer's identity that he established for Rota--Baxter algebras has one of its incarnations as the Waring formula of symmetric functions relating elementary and power sum symmetric functions (see Eq.~(\mref{eq:waring})). He also proposed to study symmetric functions in the framework of Rota--Baxter algebras.
His interest in Rota--Baxter algebras continued into the mid 1990s when he wrote, in his inspiring survey article~\cite{Ro2},
\begin{quote}
By analyzing the work of Baxter and Atkinson, I was led to conjecture that a very close relationship exists between the Baxter identity (1) and the algebra of symmetric functions.
\end{quote}
and concluded
\begin{quote}
The theory of symmetric functions of vector arguments (or Gessel functions) fits nicely with Baxter operators; in fact, identities for such functions easily translate into identities for Baxter operators. $\cdots$
In short: Baxter algebras represent the ultimate and most natural generalization of the algebra of symmetric functions.
\end{quote}

Partly motivated by Rota's enthusiasm on Rota--Baxter algebras\footnote{For example, L. Guo was introduced to the subject of the Rota--Baxter algebra by Rota's suggestion to extend a joint work of L. Guo and W. Keigher submitted to him to the context of Rota--Baxter algebras, resulting in~\cite{G-K1,G-K2} and their sequels.}, their study has experienced a remarkable renaissance this century with broad applications in mathematics and mathematical physics, in areas including combinatorics, Hopf algebras, operads, number theory, quantum field theory and classical Yang--Baxter equations~\cite{Ag,CK,EGK,G-K1,RR}. See~\cite{Gub}, as well as \cite{KRY}, for a more detailed introduction to this subject.

In particular, the connection of Rota--Baxter algebra with quasi-symmetric function envisioned by Rota was partly established by the equivalence~\cite{EG} of the mixable shuffle product introduced in~\cite{G-K1} (also known as overlapping shuffles~
\cite{Ha}) in a free commutative Rota--Baxter algebra and the quasi-shuffle product~\cite{Ho} which is known to be a generalization of quasi-symmetric functions. This realizes the algebra of quasi-symmetric functions as a large part of a free commutative Rota--Baxter algebra (on one generator).
To further develop Rota's insight, it would be useful to interpret the full free commutative Rota--Baxter algebra as a suitable generalization of quasi-symmetric functions. This is the purpose of our study. In this paper, we study free commutative nonunitary Rota--Baxter algebras. The unitary case will be considered in a separate work.

The benefit of this connection between quasi-symmetric functions and Rota--Baxter algebras is bidirectional. On the one hand, this interaction allows us to realize the abstractly defined free Rota--Baxter algebras as subalgebras of concretely defined polynomial or power series algebras. On the other hand, this interaction allows us to find generalizations of these symmetric related functions, in the spirit of Rota's aforementioned quotes. As a result, we obtain new classes of functions that share properties similar to  quasi-symmetric functions, including their canonical bases consisting of monomial and fundamental quasi-symmetric functions.

Another benefit of this connection is in the study of multiple zeta values (MZVs for short), started with
Euler and Goldbacher in the two variable case, and systematically investigated since the 1990s with the work of Hoffman \cite{Ho1992} and Zagier \cite{Zag1994}. Quasi-symmetric functions specialize to MZVs by suitable evaluations through which the multiplication of two MZVs (stuffle product) comes from the quasi-shuffle product of two monomial quasi-symmetric functions~\cite{Ho,Ho1997}. In the context of Rota--Baxter algebras, a related construction can be found in the work of Cartier~\cite{Ca} on free commutative Rota--Baxter algebras. Extending this well-known connection, we obtain from these generalized quasi-symmetric functions a class of MZVs with weights (coefficients).

The paper is organized as follows. Section~\mref{sec:quasi} is devoted to a quick review of well-known results and definitions on Rota--Baxter algebras
and quasi-symmetric functions. In particular, as motivation, we recall the connection of symmetric functions with Rota's standard Rota--Baxter algebras and, as preparation, we recall the connection of mixable shuffle product, the primary part of the free commutative Rota-Baxter algebra, with monomial quasi-symmetric functions indexed by compositions.  To put the full free commutative nonunitary Rota--Baxter algebra in the context of quasi-symmetric functions, we introduce in Section~\mref{sec:nonu} a generalization of the quasi-symmetric functions, called \eqsyms, by generalizing monomial quasi-symmetric functions to be indexed by left weak compositions. These functions, called monomial \eqsyms form a basis of the \eqsyms. A relation among monomial \eqsyms of the same degree is established by applying this construction.
In Section \ref{Sec:baseseqsym0}, we apply the concept of a $P$-partition of Stanley to define left weak fundamental quasi-symmetric functions that form the second basis for \eqsyms.
Transformation formulas between these two bases are established.  The \eqsyms permit us to define a left weak MZV and its $q$-analog in Section \ref{sec:application}.
Particular relations are established for these types of MZVs.
Especially,
a generalization of the finite double shuffle relations to the product of two left weak MZVs is provided.

{\bf Convention.} Unless otherwise specified, an algebra in this paper is assumed to be commutative, defined over a unitary commutative ring $\bfk$. Let $\NN$ and $\PP$ denote the set of nonnegative and positive integers respectively.

\section{Background}
\mlabel{sec:quasi}

To provide background and motivation for our study, we recall in this section Rota's construction of standard Rota--Baxter algebras and their relationship with symmetric functions. This is followed by the concepts of mixable shuffle product, quasi-shuffle product and their relationship with quasi-symmetric functions and MZVs.

\subsection{Rota's standard Rota--Baxter algebras and symmetric functions}

For a fixed $\lambda\in \bfk$, a {\bf Rota--Baxter $\mathbf{k}$-algebra
of weight $\lambda$} is a pair $(R,P)$ consisting of an algebra $R$ and a $\mathbf{k}$-linear
operator $P: R\rightarrow R$ that satisfies the { Rota--Baxter equation}
\begin{equation}\label{rtequ}
P(x)P(y)=P(xP(y))+P(P(x)y)+\lambda P(xy)\qquad \text{ for all } x,y\in R.
\end{equation}
Then $P$ is called a {\bf Rota--Baxter operator (RBO) of weight $\lambda$}. If $R$ is only assumed to be a nonunitary
$\mathbf{k}$-algebra, we call $R$ a nonunitary Rota--Baxter $\mathbf{k}$-algebra of weight $\lambda$.

Throughout the rest of the paper we will assume that $\lambda=1$ and drop $\lambda$ from the notations.

A Rota--Baxter algebra homomorphism is an algebra homomorphism that is compatible with the Rota--Baxter operators.
Given a commutative $\mathbf {k}$-algebra $A$ that is not necessarily unitary,
the {\bf free commutative Rota--Baxter $\mathbf {k}$-algebra on $A$} is defined to be a \rbto
$\mathbf {k}$-algebra $(F(A),P_A)$ together with a $\mathbf{k}$-algebra homomorphism
$j_A: A\rightarrow F(A)$ with the property that, for any \rbto $\mathbf{k}$-algebra $(R, P)$ and any
$\mathbf{k}$-algebra homomorphism $f: A\rightarrow R$, there is a unique Rota--Baxter $\mathbf{k}$-algebra homomorphism
$\tilde{f}: (F(A),P_A)\rightarrow (R, P)$ such that $f=\tilde{f}\circ j_A$ as $\mathbf{k}$-algebra homomorphisms.

As a motivation, we recall the first construction of free commutative Rota--Baxter algebras given by Rota~\cite{Ro1,RS}, called the standard Rota--Baxter algebra, and their relationship with symmetric functions.

Let $X$ be a given set.
For each $x\in X$, let $t^{(x)}$ denote a sequence
$( t^{(x)}_1,\cdots, t^{(x)}_n,\cdots )$
of distinct symbols $t^{(x)}_n, n\geq 1,$ such that the sets $\{t^{(x_1)}_n\}_n$ and
$\{t^{(x_2)}_n\}_n$ are disjoint for $x_1\neq x_2$ in $X$.
Denote
\[ \overline{X} = \bigcup_{x\in X} \left\{t^{(x)}_n \mid n\geq 1\right\} \]
and let $\frakA (X)=\bfk[\overline{X}]^{ \PP}$ denote the algebra of sequences with entries in
the polynomial algebra ${\bfk}[\overline{X}]$.
The addition, multiplication and scalar multiplication by
${\bfk}[\overline{X}]$ in $\frakA(X)$ are defined componentwise.

Define
\[P_X^r: \frakA(X)\to \frakA(X), \quad (a_1,a_2,a_3,\cdots)\mapsto (0,a_1,a_1+a_2,a_1+a_2+a_3,\cdots).\]
So each entry of $P_X^r (a)$ for $a=(a_1,a_2,\cdots)$
is the sum of the previous entries of $a$.
Then $P_X^r$ defines a Rota--Baxter operator on $\frakA(X)$.
The {\bf standard Rota--Baxter algebra} on $X$ is the Rota--Baxter subalgebra
$\frakS(X)$ of $\frakA(X)$ generated
by the sequences
$t^{(x)}:=(t^{(x)}_1,\cdots,t^{(x)}_n,\cdots),\ x\in X.$
An important result of Rota is
\begin{theorem} $($\cite{Ro1,RS}$)$
$(\frakS(X),P_X^r)$ is the free commutative Rota--Baxter algebra on $X$.
\mlabel{thm:rota}
\end{theorem}

Consider the special case of $X=\{x\}$. Let the sequence $(P^r_Xx)^{[n]}, n\geq 1,$ be defined by the recursion $(P^r_Xx)^{[1]}:=P^r_X(x)$ and $(P^r_Xx)^{[n+1]}:=P^r_X(x(P^r_Xx)^{[n]}), n\geq 1$.
Then
$$(P^r_Xx)^{[n]}=(0,e_n(x_1),e_n(x_1,x_2),e_n(x_1,x_2,x_3),\cdots)$$
where $e_n(x_1,\cdots,x_m)=\sum\limits_{1\leq i_1<i_2<\cdots<i_n\leq m}x_{i_1}x_{i_2}\cdots x_{i_n}$ is the elementary symmetric function
of degree $n$ in the variables $x_1,\cdots,x_m$ with the convention that $e_0(x_1,\cdots,x_m)=1$ and $e_n(x_1,\cdots,x_m)=0$ if $m<n$. Also by definition,
$$P^r_X(x^k)=(0,p_k(x_1),p_k(x_1,x_2),p_k(x_1,x_2,x_3),\cdots),$$
where $p_k(x_1,\cdots,x_m)=x_1^k+x_2^k+\cdots+x_m^k$ is the power
sum symmetric function of degree $k$ in the variables
$x_1,\cdots,x_m$. These two classes of symmetric functions are
related by Waring's formula
\begin{equation}
\exp\left (-\sum_{k=1}^\infty (-1)^kt^k p_k(x_1,\cdots,x_m)/k
    \right) = \sum_{n=0}^\infty e_n(x_1,\cdots,x_m)t^n
\text{ for all }\ m\geq 1.
\label{eq:waring}
\end{equation}
As mentioned in the introduction, Rota derived this formula as a special case of his algebraic formulation of Spitzer's identity. Cartier~\cite{Ca} gave the second construction of free commutative Rota--Baxter algebras.

\subsection{Mixable shuffle, quasi-shuffle and quasi-symmetric functions}

Another construction of free commutative \rbto $\mathbf{k}$-algebras was given by using the mixable shuffle algebra. The mixable shuffle algebra generated by a commutative (unitary or nonunitary) algebra $A$, denoted by
$MS(A)$, has its underlying module as that of the tensor algebra
\begin{align*}
T(A)=\bigoplus_{k\geq 0} A^{\otimes k}=\mathbf{k}\oplus A\oplus A^{\otimes2}\oplus\cdots,
\quad \text{ where\ } A^{\otimes k}= \underbrace{A\otimes A\otimes\cdots \otimes A}_{ k-{\text factors}},
\end{align*}
equipped  with the mixable  shuffle product $*$ defined as follows.

For pure tensors $\fraka =a_1\otimes \cdots \otimes a_m$ and $\frakb =b_1\otimes \cdots \otimes b_n$,
a {\bf shuffle} of $\fraka$ and $\frakb$ is a tensor
list from the factors of $\fraka$ and $\frakb$ in which the natural orders of the $a_i$'s and the $b_j$'s are preserved.
The {\bf shuffle product} of $\fraka$ and $\frakb$, denoted $\fraka\shap \frakb$, is the sum of all shuffles of $\fraka $ and $\frakb $.
For example, we have
$$a_1\shap (b_1\ot b_2)= a_1\ot b_1\ot b_2 + b_1\ot a_1\ot b_2 + b_1\ot b_2\ot a_1.$$

A {\bf mixable shuffle} of $\fraka$ and $\frakb$ is a shuffle of $\fraka$ and $\frakb$ where some (or none) of the pairs $a_i\ot b_j$ are replaced by $a_ib_j$. The {\bf mixable shuffle product} $\fraka \ast \frakb$ of $\fraka$ and $\frakb$ is the sum of all mixable shuffles.
For example
$$ a_1 \ast (b_1\ot b_2)=a_1\ot b_1\ot b_2 + b_1\ot a_1\ot b_2+b_1\ot b_2\ot a_1+a_1b_1\ot b_2+b_1\ot a_1b_2,$$
where $a_1b_1\ot b_2$ comes from $a_1\ot b_1\ot b_2$ by ``mixing" or merging $a_1\ot b_1$ and $b_1\ot a_1b_2$ comes from $b_1\ot a_1\ot b_2$ by ``mixing" or merging $a_1\ot b_2$. The last shuffle $b_1\ot b_2\ot a_1$ does not yield any mixed term since $a_1$ is not before any $b_j$, $j=1, 2$.

Similar to the recursive formula of the shuffle product
\begin{align*}
\fraka \shap\frakb:= a_1\otimes((a_2\otimes \cdots \otimes a_m) \shap b)+b_1\otimes(a \shap (b_2\otimes \cdots \otimes b_n)),
\quad 1\shap \fraka =\fraka \shap 1=\fraka,
\end{align*}
the mixable shuffle product can also be defined by the recursion
\begin{equation}\label{mixableshuprod}
\fraka \ast  \frakb =a_1\otimes((a_2\otimes \cdots \otimes a_m) \ast   \frakb)+b_1\otimes(\fraka \ast   (b_2\otimes \cdots \otimes b_n))
+a_1b_1\otimes((a_2\otimes \cdots \otimes a_m) \ast   (b_2\otimes \cdots \otimes b_n))
\end{equation}
with the convention that $1\ast\fraka =\fraka \ast 1=\fraka$. It is known as the {\bf quasi-shuffle}~\cite{Ho} of Hoffman.
Its equivalence to the mixable shuffle product was proved in~\cite{EG}. The mixable shuffle product equipped $MS(A)$ with a commutative algebra structure.

A {\bf composition} is a finite ordered list of positive integers. Given a composition
$\alpha=(\alpha_1,\cdots,\alpha_k)$, we define its {\bf weight} or {\bf size} to be $|\alpha|=\alpha_1+\cdots+\alpha_k$ and its {\bf length} to be
$\ell(\alpha):=k$. We call $\alpha_i, 1\leq i\leq k,$ the {\bf components} of $\alpha$. If $|\alpha|=n$, then we say $\alpha$ is a composition of $n$ and write $\alpha\models n$. For convenience we denote by $\emptyset$ the unique composition whose weight and length are $0$, called the {\bf empty} composition.
There is a natural bijection between compositions of size $n$ and subsets of $[n-1]=\{1,\cdots,n-1\}$ which maps a composition
$\alpha=(\alpha_1,\cdots,\alpha_k)\models n$ to the set of its partial sums, not including $n$ itself, that is,
$$
\set(\alpha):=\{\alpha_1,\alpha_1+\alpha_2,\cdots,\alpha_1+\alpha_2+\cdots+\alpha_{k-1}\}.
$$
Given compositions $\alpha$ and $\beta$ of $n$, we say that $\alpha$ is a {\bf refinement} of $\beta$ (or $\beta$ is a {\bf coarsening} of $\alpha$), denoted
$\alpha\preceq\beta$, if summing some consecutive components of $\alpha$ gives $\beta$. For example, $(1,2,1,4,1)\preceq(3,1,5)$.

Let $\mathfrak{X}=\{x_1,x_2,\cdots\}$ be a countably infinite totally ordered set of commuting variables.
Recall that a {\bf quasi-symmetric function} is a formal power series of finite degree $f\in \bfk[[\mathfrak{X}]]$ such that
for each composition $\alpha=(\alpha_1,\cdots,\alpha_k)$, all monomials $x^{\alpha_1}_{n_1}\cdots x^{\alpha_k}_{n_k}$
in $f$ with indices $1\leq n_1 < \cdots< n_k$ have the same coefficient.
The set of all quasi-symmetric functions, denoted $\qsym$, is a graded algebra with function multiplication, and with grading
$$
\qsym:=\bigoplus_{n\geq 0}\qsym_n,
$$
where $\qsym_0$ is spanned by $M_\emptyset=1$ and all other $\qsym_n$ are spanned by $\{M_{\alpha}\}_{\alpha\models n}$ where
$$M_{\alpha}:=\sum_{1\leq n_1<\cdots<n_k} x_{n_1}^{\alpha_1}\cdots x_{n_k}^{\alpha_k}$$
for a composition $\alpha=(\alpha_1,\cdots,\alpha_k)$ of $n$.
This basis is called the basis of {\bf monomial quasi-symmetric functions}.
A second basis for $\qsym$ consists of  {\bf fundamental quasi-symmetric functions}, also indexed by compositions. They are defined by
$F_\emptyset=1$ and
$$
F_{\alpha}=\sum_{\beta\preceq\alpha}M_{\beta}.
$$
For instance, $F_{(2,1)}=M_{(1,1,1)}+M_{(2,1)}$.
Moreover, $\qsym_n={\rm span}\{M_\alpha|\alpha\models n\}={\rm span}\{F_\alpha|\alpha\models n\}$.

The multiplication of two monomial quasi-symmetric functions is dictated by the quasi-shuffle product, or mixable shuffle product, on the indexing compositions.
This is a special case of the quasi-shuffle product defined in Eq.\ \eqref{mixableshuprod}.
In the special case when the algebra $A$ is the nonunitary algebra $x\bfk[x]$ of polynomials in one variable without the constant term, a basis of $MS(A)$ consists of pure tensors $x^{\alpha}:=x^{\alpha_1}\ot \cdots \ot x^{\alpha_k}$ where $\alpha=(\alpha_1,\cdots, \alpha_k)$ is a composition, together with the identity $1$. Through the natural linear bijection
$$\rho: \bfk\calc\to MS(A), \quad \alpha\mapsto x^\alpha,$$
on the space spanned by the set $\calc$ of compositions, the quasi-shuffle product $\ast$ on $MS(A)$ defines a product,
also denoted $\ast$, on $\bfk\calc$ by transport of structures. More precisely, if $\alpha=(\alpha_1,\cdots,\alpha_m)$,
$\beta=(\beta_1,\cdots,\beta_n)$, then we have the recursion
\begin{align}\label{compqsh}
\alpha\ast\beta=\left(\alpha_1,(\alpha_2,\cdots,\alpha_m)\ast\beta\right) \!+\!\left(\beta_1,\alpha\ast(\beta_2,\cdots,\beta_n)\right)
\!+\!\left(\alpha_1+\beta_1,(\alpha_2,\cdots,\alpha_m)\ast(\beta_2,\cdots,\beta_n)\right)
\end{align}
with the convention that $1\ast\alpha =\alpha \ast 1=\alpha$.

\begin{convention} To simplify notations in the rest of the paper, we will adapt the following convention. Let $\bfk X$ be the free module on a set $X$ and let $\cdot: X\times X \to \bfk X$ be a multiplication that extends to one on $\bfk X$. Then for $x, y\in X$, we have $x \cdot y = \sum\limits_{z\in X}c_z z$ which we denote by $\sum\limits_{z\in (x\cdot y)} z$. Thus $V$ is a $\bfk$-algebra whose multiplication is suppressed. Then a linear map $f: \bfk X\to V$ preserves the multiplications: $f(x)f(y)=f(x\cdot y)$ if and only if $f(x)f(y)=\sum\limits_{u\in (x\cdot y)} f(u)$. Note that here $u\in (x\cdot y)$ does not mean that $u$ is a basis element in $X$.
\label{con:index}
\end{convention}

In particular, we let $\alpha\ast\beta =\sum\limits_{\gamma\in (\alpha\ast\beta)} \gamma$ denote the linear combination resulting from $\alpha\ast \beta$.
By~\cite{Ho} we obtain an algebra isomorphism of $MS(A)$ with
the algebra $\qsym$
by sending $x^{\alpha}$ to $M_\alpha$.
Here the multiplication on \qsym\ is the natural multiplication of power series. Therefore, given two compositions $\alpha$ and $\beta$, we have
$$M_\alpha M_\beta=\sum_{\gamma\in (\alpha\ast \beta)} M_{\gamma},$$
as in Eq.~\eqref{compqsh}. In other words, let
\begin{equation}
\phi: \bfk \calc \to \qsym, \quad \alpha \mapsto M_\alpha, \alpha\in \calc.
\label{eq:phi}
\end{equation}
Then
$$M_\alpha M_\beta = \phi(\alpha\ast \beta).$$

Evaluating $x_n$ at $n^{-1}$, $M_\alpha$ gives the MZV
\begin{equation}\label{eq:mzv}
\zeta(\alpha):=\zeta(\alpha_1,\cdots,\alpha_k):=\sum_{1\leq n_1<\cdots<n_k} \frac{1}{n_1^{\alpha_1}\cdots n_k^{\alpha_k}},
\end{equation}
which converges for $\alpha_j\geq 1, 1\leq j\leq k$ and $\alpha_k\geq 2$. The quasi-shuffle relation or stuffle relation of MZVs is a direct consequence of the quasi-shuffle product on \qsym.

\section{Free nonunitary Rota--Baxter algebras and \eqsyms}
\label{sec:nonu}
In this section, we first outline a construction for free commutative nonunitary Rota--Baxter algebras on one generator by making use of the mixable shuffle algebras recalled in the last section. We then introduce a class of power series that on the one hand contains quasi-symmetric functions with similar properties and on the other hand is isomorphic to the free nonunitary Rota--Baxter algebra on one generator.

Let $A$ be a commutative unitary $\bfk$-algebra. Define
\begin{align*}
\sha_{\bfk}(A):=A\otimes MS(A)=\bigoplus_{k\geq1} A^{\otimes k}
\end{align*}
to be the tensor product of the algebras $A$ and $MS(A)$.
More precisely, for two pure tensors $a_0\otimes \fraka =a_0\otimes a_1\otimes \cdots \otimes a_m$
 and $b_0\otimes \frakb =b_0\otimes b_1\otimes \cdots \otimes b_n$, the {\bf augmented mixable shuffle product} $\diamond$ on $\sha_{\bfk}(A)$ is defined by
\begin{align}\label{defdiamond}
(a_0\otimes \fraka )\diamond (b_0\otimes \frakb ):=\begin{cases}
a_0b_0,&m=n=0,\\
(a_0b_0)\otimes \fraka, & m>0,n=0, \\
(a_0b_0)\otimes \frakb, & m=0,n>0,\\
(a_0b_0)\otimes(\fraka \ast   \frakb ),&m>0,n>0.
\end{cases}
\end{align}

\begin{theorem}\cite{G-K1}\label{freerbamsp}
The algebra $(\sha_{\bfk}(A),\diamond)$, with the linear
operator $P_A :\sha_{\bfk}(A)\rightarrow \sha_{\bfk}(A)$ sending $\fraka$ to $1\otimes \fraka$, is the free commutative \rbto algebra generated by $A$.
\end{theorem}

Now let $A$ be a commutative nonunitary $\bfk$-algebra and let $\tilde{A}=\bfk\oplus A$ be the unitarization of $A$. Define
\begin{align*}
\sha_{\bfk}(A)^0:=\bigoplus_{k\geq 0} (\tilde{A}^{\otimes k}\otimes A)
\end{align*}
with the convention that $\tilde{A}^{\otimes 0}=\bfk$ and hence $\tilde{A}^{\otimes 0}\otimes A=A$.
Then $\sha_{\bfk}(A)^0$ is a $\bfk$-submodule of $\sha_{\bfk}(\tilde{A})$, additively spanned by pure tensors of the form
\begin{align*}
a_0\otimes\cdots\otimes a_n,\qquad a_i\in \tilde{A},\ 0\leq i\leq n-1,\ a_n\in A.
\end{align*}
With the restriction of $P_{\tilde{A}}$, denoted by $P_A$, $\sha_{\bfk}(A)^0$ is a subobject of $\sha_{\bfk}(\tilde{A})$ in the category of
commutative nonunitary \rbto algebras. By \cite[Proposition 2.6]{G-K2}, $(\sha_{\bfk}(A)^0, P_A)$ is the free
commutative nonunitary \rbto algebra generated by $A$ characterized by its usual universal property.

We now consider the special case when $A =\bfk[x]$ or $A=x\bfk[x]$. Then Theorem \ref{freerbamsp} can be restated as
\begin{theorem}
The $\bfk$-module
\begin{align}\label{shaklambdax11}
\sha(x):=\sha_{\bfk}\left(\bfk[x]\right)=\bigoplus_{k\geq 0}  \bfk[x]^{\otimes (k+1)}
=\bigoplus_{\alpha_i\geq0, 0\leq i\leq k}\bfk x^{\alpha_0}\otimes\cdots\otimes x^{\alpha_k},
\end{align}
with the product in Eq.~\eqref{defdiamond} and the operator $P_x: \fraka\mapsto 1\ot \fraka$, is the free unitary Rota--Baxter algebra on $x$. The restriction of the product and operator on the submodule
\begin{align}\label{shaklambdax10}
\sha(x)^0:=\sha_{\bfk}(x\bfk[x])^0=\bigoplus_{k\geq 0} \left(\bfk[x]^{\otimes k}\otimes x\bfk[x]\right)
=\bigoplus_{\alpha_i\geq0, 0\leq i\leq k-1,\alpha_k\geq 1}\bfk x^{\alpha_0}\otimes\cdots\otimes x^{\alpha_k}
\end{align}
gives the free nonunitary \rbto algebra on $x$.
\mlabel{freenonurba}
\end{theorem}

As recalled in Section~\mref{sec:quasi}, the mixable shuffle algebra
$$MS(x\bfk[x]) =\bfk 1\bigoplus \bfk \{x^{i_1}\ot \cdots \ot x^{i_k}\,|\, (i_1,\cdots,i_k)\in \PP^k, k\geq 1\}$$
is isomorphic to the algebra of quasi-symmetric functions \qsym\ in the formal power series $\bfk[[\mathfrak{X}]]$.
Let
$$MS(x\bfk[x])^0:= \bfk\{x^{i_1}\ot \cdots \ot x^{i_k}\,|\, (i_1,\cdots,i_k)\in \NN^{k-1}\times \PP,k\geq 1\}.$$
Then, from the above theorem, we have
$\sha(x)^0=MS(x\bfk[x])^0$
which is larger than $MS(x\bfk[x])$.

In order to realize $\sha(x)^0$ as formal power series, we first need to do so for $MS(x\bfk[x])^0$.
To this end, we first generalize the definition of quasi-symmetric functions, by generalizing the concept of compositions.

A {\bf weak composition} \cite{Sta} $\alpha=(\alpha_1,\cdots,\alpha_k)$ of a nonnegative integer $n$, denoted $\alpha\wcomp n$, is a sequence of nonnegative integers whose sum is $n$. We call the $\alpha_i$ the {\bf components} of $\alpha$ and $\ell(\alpha):=k$ the {\bf length} of $\alpha$.
If $\alpha\wcomp n$, then we call $n$ the {\bf size} of $\alpha$, denoted $|\alpha|$.

\begin{defn}
A {\bf left weak composition}\footnote{Relating weak compositions to Rota-Baxter algebras needs a different strategy and will be considered in a separate work.}
is a weak composition that ends with a positive component, that is, if $\alpha=(\alpha_1,\cdots,\alpha_k)$ with $\alpha_k$ being a positive integer. Let $\LWC$ denote the set of left weak compositions.
\end{defn}

For a left weak composition $\alpha$, it is convenient to group blocks of zeros together and denote
\begin{align*}
\alpha=(0^{i_1}, s_1,0^{i_2}, s_2,\cdots,0^{i_{j}},s_{j}),\quad s_p\geq1, i_{p}\geq0, p=1,2,\cdots,j,
\end{align*}
where $0^i$ means a string of $i$ zero components. We also write
$\alpha=0^{i_1}s_10^{i_2}s_2\cdots0^{i_j}s_j$ for convenience.

Recall that QSym can be equivalently defined to be the subspace of formal power series in the countable variables $\frakX:=\{x_i\,|\, i\geq 1\}$ spanned by the monomial quasi-symmetric functions $M_\alpha$, as $\alpha$ run through all the compositions.
We extend this notion for left weak compositions as follows. 

\begin{defn}
we call a formal power series of finite degree $f\in \bfk[[\mathfrak{X}]]$  a {\bf left weak composition quasi-symmetric function}, or simply an {\bf \lwc quasi-symmetric function}, if
$f$ is a linear combination of the following formal power series \begin{equation}
M_\alpha:=\sum_{1\leq n_1<\cdots<n_k} x_{n_1}^{\alpha_1}\cdots x_{n_k}^{\alpha_k},
\mlabel{eq:eqsym}
\end{equation}
indexed by left weak compositions $\alpha=(\alpha_1,\cdots,\alpha_k)$.
We call the $M_\alpha$ the
{\bf \lwc monomial quasi-symmetric function} and denote $\EQSYM$ the space of \eqsyms.
\end{defn}

With the convention that $\binom{m}{n}=0$ if $m<n$, we have the following formula for an \lwc monomial quasi-symmetric function $M_{\alpha}$.
\begin{lemma}\label{lemmalphamon}
Let $\alpha=(0^{i_1}, s_1,0^{i_2}, s_2,\cdots,0^{i_{k}},s_{k})$ be a left weak composition, where $s_p\geq1, i_{p}\geq0, p=1,2,\cdots,k$. Then
\begin{align}\label{Malphas}
M_\alpha=\sum_{1\leq n_1<n_2<\cdots<n_k}\binom{n_1-1}{i_1}\binom{n_2-n_1-1}{i_2}\cdots\binom{n_k-n_{k-1}-1}{i_k}
x_{n_1}^{s_1}x_{n_2}^{s_2}\cdots x_{n_k}^{s_k}.
\end{align}
\end{lemma}
\begin{proof}
By Eq. \eqref{eq:eqsym}, we know
\begin{align*}
M_\alpha=\sum_{1\leq n_{11}<\cdots<n_{1i_1}<n_1<\cdots<n_{k1}<\cdots<n_{ki_{k}}<n_k}
x_{n_{11}}^0\cdots x_{n_{1i_1}}^0 x_{n_1}^{s_1}x_{n_{21}}^0\cdots x_{n_{2i_2}}^0 x_{n_2}^{s_2}\cdots x_{n_{k1}}^0\cdots x_{n_{ki_k}}^0 x_{n_k}^{s_k}.
\end{align*}
For fixed $n_{p-1}$ and $n_p$, $p=1,\cdots, k$, with the convention that $n_0=0$, one can choose $n_{p1}, \cdots, n_{pi_p}$ satisfying $n_{p-1}<n_{p1}<\cdots<n_{p\,i_p}<n_{p}$
in $\binom{n_{p}-n_{p-1}-1}{i_{p}}$ ways. Note that $\binom{m}{n}=0$ if $m<n$ and $x^0=1$ for all indeterminants $x$. Then Eq.~\eqref{Malphas} follows.
\end{proof}
For example,
\begin{align*}
M_{(0^2,3)}=\sum_{1\leq n_1<n_2<n_3}x_{n_1}^0x_{n_2}^0x_{n_3}^3=\sum_{n\geq1}\binom{n-1}{2}x_{n}^3=\sum_{n\geq3}\binom{n-1}{2}x_{n}^3.
\end{align*}
Clearly, $M_\alpha$ is a quasi-symmetric function if and only if $i_{p}= 0$ for all $1\leq {p}\leq k$.

To make the parallel to monomial quasi-symmetric functions complete, we have
\begin{lemma}\label{qsym0basis}
The set $\{M_\alpha|\alpha\in \NN^{k-1}\times \PP, k\geq1\}$ is a $\bfk$-basis for $\EQSYM$.
\end{lemma}
\begin{proof}
It suffices to show that $M_\alpha$ for any finitely many distinct $\alpha$ are linear independent.
Let $\alpha_j=(0^{i_{j1}},s_{j1},0^{i_{j2}},s_{j2},\cdots,0^{i_{jk_j}},s_{jk_j})$ be distinct left weak compositions,
$j=1,\cdots,n$. Assume that $c_1M_{\alpha_1}+\cdots+c_nM_{\alpha_n}=0$ is a non-zero combination with $c_j\neq 0$,  $j=1,\cdots,n$.
Let $r_j=\ell(\alpha_j)$, $j=1,\cdots,n$.
By permuting $r_1,\cdots,r_n$ if necessary, we may assume  $r_1=\min\{r_1,\cdots,r_{n}\}$.
By Lemma \ref{lemmalphamon}, the coefficient of
$x_{i_{j1}+1}^{s_{j1}}x_{i_{j1}+i_{j2}+2}^{s_{j2}}\cdots x_{r_j}^{s_{jk_j}}$ in $M_{\alpha_j}$ is $1$.
Since $\alpha_j$ are different from each other, the coefficient of
$x_{i_{11}+1}^{s_{11}}x_{i_{11}+i_{12}+2}^{s_{12}}\cdots x_{r_1}^{s_{1k_1}}$ in
$M_j$ is $\delta_{1j}$, where $\delta_{ij}$ is the Kronecker delta.
Thus,  the coefficient of the above monomial in
$c_1M_{\alpha_1}+\cdots+c_nM_{\alpha_n}$ is $c_1$ and hence $c_1$ must be $0$, a contradiction.
This completes the proof.
\end{proof}

Analogous to Eq.~\eqref{compqsh} we define the quasi-shuffle product, also denoted by $\ast$, for two left weak compositions, and denote it by $\alpha \ast \beta=\sum\limits_{\gamma\in (\alpha\ast\beta)}\gamma$ with the notation in Convention~\ref{con:index}.
Then we have the following multiplication formula for \lwc monomial quasi-symmetric functions.

\begin{theorem}
Let $\alpha$ and $\beta$ be left weak compositions. Then
 $M_\alpha M_\beta=\sum\limits_{\gamma\in (\alpha\ast\beta)} M_{\gamma}$ with the notation in Convention~\ref{con:index}. In other words, defining
\begin{equation}
 \phi: \bfk \LWC \to \EQSYM, \quad \alpha\mapsto M_\alpha,
 \mlabel{eq:ephi}
 \end{equation}
extending the linear bijection $\phi$ in Eq.~(\ref{eq:phi}), then $M_\alpha M_\beta = \phi(\alpha\ast \beta).$
\label{thmmonproduct}
\end{theorem}
\begin{proof}
Let $\alpha=(\alpha_1,\cdots,\alpha_k)\in \NN^{k-1}\times \PP$ and $\beta=(\beta_1,\cdots,\beta_l)\in \NN^{l-1}\times \PP$ be two left weak compositions.
It is easy to see that every summand of $\alpha*\beta$ is also a left weak composition.
Here we use the convention that $M_\alpha=1$ if $\ell(\alpha)=0$. We apply the induction on the sum $k+l$ of the lengths of $\alpha$ and $\beta$.
If $k+l=1$, then either $M_\alpha$ or $M_{\beta}$ is $1$, so the desired result is clearly true. Assume that it has been verified for
$k+l\leq t$ for a given $t\geq 1$ and consider the case $k+l=t+1$.
Then, by Eq. \eqref{eq:eqsym},
\begin{align*}
M_\alpha M_\beta=\sum_{\stackrel{1\leq n_1<n_2<\cdots<n_k}{1\leq m_1<m_2<\cdots<m_l}}
x_{n_1}^{\alpha_1}\cdots x_{n_k}^{\alpha_k}x_{m_1}^{\beta_1}\cdots x_{m_l}^{\beta_l}.
\end{align*}
The sum on the right hand can be divided into three parts:

(a) if $n_1<m_1$, then, by the induction hypothesis,
$$\sum_{1\leq n_1}x_{n_1}^{\alpha_1}
\left(\sum_{n_1<n_2<\cdots<n_k}x_{n_2}^{\alpha_2}\cdots x_{n_k}^{\alpha_k}\sum_{n_1< m_1<m_2<\cdots<m_l}x_{m_1}^{\beta_1}\cdots x_{m_l}^{\beta_l}\right)
=M_{(\alpha_1,(\alpha_2,\cdots,\alpha_k)*\beta)},$$

(b) if $m_1<n_1$, then, by the induction hypothesis,
$$\sum_{1\leq m_1}x_{m_1}^{\beta_1}
\left(\sum_{m_1<n_1<\cdots<n_k}x_{n_1}^{\alpha_1}\cdots x_{n_k}^{\alpha_k}\sum_{m_1< m_2<\cdots<m_l}x_{m_2}^{\beta_2}\cdots x_{m_l}^{\beta_l}\right)
=M_{(\beta_1,\alpha*(\beta_2,\cdots,\beta_l))},
$$

(c) if $m_1=n_1$, then, by the induction hypothesis,
$$\sum_{1\leq n_1=m_1}x_{n_1}^{\alpha_1+\beta_1}
\left(\sum_{n_1<n_2<\cdots<n_k}x_{n_2}^{\alpha_2}\cdots x_{n_k}^{\alpha_k}\sum_{n_1< m_2<\cdots<m_l}x_{m_2}^{\beta_2}\cdots x_{m_l}^{\beta_l}\right)
=M_{(\alpha_1+\beta_1,(\alpha_2,\cdots,\alpha_k)*(\beta_2,\cdots,\beta_l))}.
$$
The multiplication of two left weak compositions satisfies the same recursion as in Eq.~(\ref{compqsh}). Thus combining the three parts , we obtain $M_\alpha M_\beta=\sum\limits_{\gamma\in (\alpha\ast\beta)} M_{\gamma}$, as desired.
\end{proof}

By Theorem \ref{thmmonproduct},
the space of \eqsyms is a graded algebra
\begin{equation*}
\EQSYM:=\bigoplus_{n\geq 0}\EQSYM_n,
\end{equation*}
where $\EQSYM_0$ is spanned by $M_\emptyset=1$ and $\EQSYM_n$ for $n\geq 1$
consists of homogeneous \eqsyms of degree $n$. Here the degree of an \lwc monomial quasi-symmetric function $M_\alpha$ is defined to be the size of $\alpha$.
Note that since we allow the entries of $\alpha$ to be $0$, the rank of the homogeneous component $\EQSYM_n$ is infinite for each $n\geq 1$.

For a vector $\alpha:=(\alpha_0,\cdots,\alpha_k)\in \NN^{k}\times \PP$, denote $\alpha^\prime:=(\alpha_1,\cdots,\alpha_k)\in \NN^{k-1}\times \PP$ and $\alpha=(\alpha_0,\alpha^\prime)$. Let $\bar{\mathfrak{X}}=\{x_0\}\cup\mathfrak{X}$.  In the power series algebra $\bfk[[\overline{\mathfrak{X}}]]$, define
\begin{equation}
\overline{M}_{\alpha}:=x_0^{\alpha_0} M_{\alpha^\prime} =
x_0^{\alpha_0}\sum_{1\leq n_1<\cdots<n_k}x_{n_1}^{\alpha_1}\cdots x_{n_k}^{\alpha_k}.
\mlabel{eq:eeqsym}
\end{equation}

Define an augmented mixable shuffle product for left weak compositions by the same recursion as Eq.~\eqref{defdiamond}. More precisely,
if $(\alpha_0,\alpha')=(\alpha_0,\alpha_1,\cdots,\alpha_m)$ and $(\beta_0,\beta')=(\beta_0,\beta_1,\cdots,\beta_n)$ are left weak compositions,
then we define their {\bf augmented mixable shuffle product} by
\begin{align}\label{augmixshuweakcomp}
(\alpha_0,\alpha')\diamond(\beta_0,\beta')
:=\left\{\begin{array}{ll}
(\alpha_0+\beta_0),&m=n=0\\
(\alpha_0+\beta_0, \alpha'), & m>0,n=0 \\
(\alpha_0+\beta_0,\beta'), & m=0,n>0\\
(\alpha_0+\beta_0,\alpha'\ast \beta' ),&m>0,n>0
\end{array}\right \}
=\sum_{\gamma\in (\alpha\diamond\beta)} \gamma
\end{align}
with the notation in Convention~\ref{con:index}.
This defines a multiplication on the space spanned by left weak compositions.

\begin{lemma}
Let $\alpha=(\alpha_0,\alpha^\prime)$ and $\beta=(\beta_0,\beta')$ be weak compositions. Then
\begin{align*}
\overline{M}_{\alpha}\overline{M}_{\beta}=\sum_{\gamma\in (\alpha\diamond \beta)} \overline{M}_{\gamma},
\end{align*}
with the notation in Convention~\ref{con:index}.
\label{lem:eqsymmult}
\end{lemma}
\begin{proof}
By Eq.~\eqref{eq:eeqsym}, we have
$\overline{M}_{\alpha}\overline{M}_{\beta}=x_0^{\alpha_0} M_{\alpha^\prime}x_0^{\beta_0} M_{\beta^\prime}
=x_0^{\alpha_0+\beta_0} M_{\alpha^\prime} M_{\beta^\prime}$. It follows from Eq.~\eqref{augmixshuweakcomp} and Theorem~\ref{thmmonproduct} that
\begin{align*}
\overline{M}_{\alpha}\overline{M}_{\beta}
=x_0^{\alpha_0+\beta_0} M_{\alpha^\prime*\beta^\prime}
=\overline{M}_{(\alpha_0+\beta_0,\alpha^\prime*\beta^\prime)}
=\sum\limits_{\gamma\in (\alpha\diamond \beta)}\overline{M}_{\gamma},
\end{align*}
as desired.
\end{proof}

Thus we have identified the free nonunitary Rota--Baxter algebra of weight $1$ on $x$ with the algebra of \eqsyms and thus identified this free Rota-Baxter algebra as a subalgebra of $\bfk[[\overline{\mathfrak{X}}]]$. This result goes in the same spirit as Theorem~\ref{thm:rota} of Rota relating free Rota-Baxter algebras with symmetric functions.

Let $\overline{{\EQSYM}}$ be the $\bfk$-submodule of $\bfk[[\overline{\mathfrak{X}}]]$ spanned by $\overline{M}_{\alpha}$,
where $\alpha\in \NN^{k}\times \PP$, $k\geq0$.
Define a $\bfk$-linear operator
\begin{equation}
P_Q:\overline{\EQSYM}\longrightarrow \overline{\EQSYM}, \quad P_Q(\overline{M}_{\alpha})=\overline{M}_{(0,\alpha)}
\label{eq:pq}
\end{equation}
and then extend linearly.

\begin{theorem}\label{nonunitarymainth}
The linear operator $P_Q$ is a Rota--Baxter operator of weight $1$ on $\overline{\EQSYM}$,
and $(\overline{{\EQSYM}}, P_Q)$ is the free nonunitary commutative Rota--Baxter algebra  of weight $1$ on one generator.
\end{theorem}

\begin{proof}
By Theorem~\mref{freenonurba}, as a free $\mathbf{k}$-module, $\sha(x)^0$ has the $\mathbf{k}$-linear basis
\begin{align*}
\{x^{\alpha_0}\otimes x^{\alpha_1}\otimes\cdots\otimes x^{\alpha_k}|\alpha_i\in \mathbb{N}, \alpha_k\in \mathbb{P}, 0\leq i\leq k-1,k\in \mathbb{N}\}.
\end{align*}
By Lemma \ref{qsym0basis}, $\overline{\EQSYM}$ has the $\mathbf{k}$-linear basis
\begin{align*}
\{{\bar{M}}_{(\alpha_0,{\alpha_1},\cdots,{\alpha_k})}|\alpha_i\in \mathbb{N}, \alpha_k\in \mathbb{P}, 0\leq i\leq k-1,k\in \mathbb{N}\}.
\end{align*}
Therefore, the assignment
$$ \varphi:\sha(x)^0\longrightarrow \overline{\EQSYM}, \quad
x^{\alpha_0}\otimes x^{\alpha_1}\otimes\cdots\otimes x^{\alpha_k}
\mapsto {\bar{M}}_{(\alpha_0,{\alpha_1},\cdots,{\alpha_k})}$$
defines a linear bijection. Further, by Eqs.~\eqref{defdiamond}, \eqref{augmixshuweakcomp}, \eqref{eq:pq} and  Lemma~\ref{lem:eqsymmult}, $\varphi$ is a Rota--Baxter algebra isomorphism. Thus the conclusion follows from Theorem~\ref{freenonurba}.
\end{proof}

Using the classical case of Spitzer's identity, Theorem \ref{nonunitarymainth} gives a relation among the monomial quasi-symmetric functions of the same degree.

Recall that if $(R,P)$ is a commutative Rota--Baxter $\QQ$-algebra of weight $1$, then for $b\in R$, we have the following Spitzer's identity~\cite{RS} in the ring of power series $R[[t]]$:
\begin{align}\label{sptizexpand2}
\exp\left(P\left(\log(1+bt)\right)\right)
=\sum_{n=0}^{\infty}(Pb)^{[n]}t^n,
\end{align}
where
\begin{align*}
(Pb)^{[n]}=\underbrace{P(b(P(b\cdots(P(b))\cdots)))}_{n\ {\rm iteration}}
\end{align*}
with the convention that $(Pb)^{[1]}=P(b)$ and $(Pb)^{[0]}=1$.

\begin{prop}\label{proprelacommon}
Let $n,k$ be positive integers. We write $k^n$ for the composition $(k,\cdots,k)$ where $k$ appears $n$ times.
Then
\begin{align*}
M_{k^n}=(-1)^n\sum_{\alpha\models n}\frac{(-1)^{\ell(\alpha)}M_{k{\alpha_1}}\cdots M_{k{\alpha_{\ell(\alpha)}}}}{\ell(\alpha)!\alpha_1\cdots\alpha_{\ell(\alpha)}}.
\end{align*}
\end{prop}
\begin{proof}
It follows from Eq.~\eqref{sptizexpand2} that
\begin{align*}
\exp\left(-\sum_{i=1}^\infty\frac{(-1)^i}{i}P(b^i)t^i\right)=\sum_{n=0}^{\infty}(Pb)^{[n]}t^n.
\end{align*}
By a direct calculation we obtain
\begin{align*}
\sum_{n=0}^{\infty}(Pb)^{[n]}t^n=&\sum_{j=0}^\infty \frac{(-1)^j}{j!}\left(\sum_{i=1}^\infty\frac{(-1)^i}{i}P(b^i)t^i\right)^j\cr
=&\sum_{n=0}^{\infty}\left((-1)^n\sum_{\alpha\models n}\frac{(-1)^{\ell(\alpha)}P(b^{\alpha_1})\cdots P(b^{\alpha_{\ell(\alpha)}})}{\ell(\alpha)!\alpha_1\cdots\alpha_{\ell(\alpha)}}\right)t^n.
\end{align*}

Take $b={\bar{M}}_{k}$ in the free nonunitary Rota--Baxter algebra $\overline{\EQSYM}$.
Then $P(b^i)=P({\bar{M}}_{k}^i)=P({\bar{M}}_{ki})=M_{ki}$ and
$(Pb)^{[n]}=M_{k^n}$.
Consequently,  the desired conclusion follows.
\end{proof}

\section{$P$-partitions and \lwc fundamental quasi-symmetric functions}\label{Sec:baseseqsym0}

In this section, we use the concept of a $P$-partition to generalize fundamental quasi-symmetric functions for compositions to left weak compositions. We show that \lwc fundamental quasi-symmetric functions, as well as the \lwc monomial quasi-symmetric functions,
form a $\bfk$-basis for $\EQSYM$. We also give the transformation formulas between these two bases.
We assume that $\bfk=\QQ$ in this section.

\subsection{$P$-partitions and \eqsyms}

$P$-partitions is an important combinatorial tool for finding the number
of permutations in a given set of permutations with a given descent set.
The definition of $P$-partitions is due to Stanley \cite{Sta}. Gessel \cite{Ge} applied this concept to connect permutation descents and  quasi-symmetric functions.
In this subsection, we use the $P$-partition to give a combinatorial meaning for \eqsyms.

A partial order on a set of  $n$ positive integers is called a {\bf labeled poset}. These integers are referred to as the {\bf labels}
of the poset. Following the notation of \cite{Ge}, we write the set of labels as $[n]=\{1,2,\cdots,n\}$, $<$ for the usual total order
 on $[n]$. Let $P$ denote the set $[n]$ equipped with a partial order $<_P$.
We use the same symbol to refer to both the poset and its set of labels when the meaning is clear from the context.

A {\bf P-partition} is a function $f:P\rightarrow \PP$ satisfying
\begin{enumerate}
\item  $a<_Pb$ implies $f(a)\leq f(b)$, and
\item $a<_Pb$ and $a>b$ imply $f(a)<f(b)$.
\end{enumerate}
We denote the set of all P-partitions by $A(P)$.

Given two disjoint posets $P$ and $Q$,
the {\bf ordinal sum} of $P$ and $Q$ is the poset $P\oplus Q$ on the union
$P\cup Q$, such that all of the order relations of $P$ and $Q$ are retained, and in addition, $p<q$ for all $p\in P$ and $q\in Q$.

\begin{const}\label{constp-partextend}
Let $\alpha=(0^{i_1},s_1, \cdots,0^{i_k},s_k)$ be a left weak composition, where $i_p\geq0$, $s_p>0$,
$p=1,\cdots,k$. Construct a chain $P_\alpha:=C_1 \oplus P_1\oplus \cdots\oplus C_{k}\oplus P_k$ on $i_1+s_1+\cdots+i_k+s_k$
elements, where $C_p$ and $P_p$ are chains with $i_p$ and $s_p$ elements respectively.
Make $P_\alpha$ into a labeled poset by numbering first the elements of $C_1,C_2,\cdots,C_k$ (in the designated order) and then the elements of $P_k,P_{k-1},\cdots,P_1$ (in the designated order), such that the labeled order is compatible with the order of $C_p$ and $P_p$.
\end{const}

For example,
let $\alpha=(0^2,2,0,1,3)$. Then $C_3=\emptyset$ and $P_\alpha=\{1,2\}\oplus\{8,9\}\oplus\{3\}\oplus\{7\}\oplus\emptyset\oplus\{4,5,6\}$.

For any  $P$-partition $f\in A(P_\alpha)$, define the {\bf weight} of $f$ to be the monomial
\begin{align*}
w(f)=\prod_{q\in P_\alpha} x_{f(q)}^{\sigma_q},\qquad {\rm where}\qquad
\sigma_q=\begin{cases}
0, &q\in C_p\ {\rm for\ some}\ p,\\
1, &q\in P_p\ {\rm for\ some}\ p,
\end{cases}
\end{align*}
and define the {\bf generating function} for $P_\alpha$ to be
\begin{align*}
\Gamma(P_\alpha)=\sum_{f\in A(P_\alpha)}w(f).
\end{align*}

There is an explicit expression for $\Gamma(P_\alpha)$.

\begin{lemma}\label{lemfundqusm}
Let  $\alpha=(0^{i_1},s_1,\cdots,0^{i_k},s_k)$ be a  left weak composition and let $a_p=i_1+s_1+\cdots+i_p+s_p$, $p=1,2,\cdots,k$. Then
\begin{align*}
\Gamma(P_\alpha)=\sum x_{n_1}^0\cdots x_{n_{i_1}}^0 x_{n_{i_1+1}}\cdots x_{n_{a_1}}
\cdots x_{n_{a_{k-1}+1}}^0\cdots x_{n_{a_{k-1}+i_{k}}}^0 x_{n_{a_{k-1}+i_{k}+1}}\cdots x_{n_{a_k}},
\end{align*}
where the sum is over all variables in $\mathfrak{X}$ subject to the condition
$$1\leq n_1\leq \cdots \leq n_{a_1}<n_{a_1+1}\leq\cdots< n_{a_{k-1}+1}\leq\cdots \leq n_{a_{k}}.$$
\end{lemma}

Continue with the above example of $\alpha=(0^2,2,0,1,3)$ and $P_\alpha=\{1,2\}\oplus\{8,9\}\oplus\{3\}\oplus\{7\}\oplus\emptyset\oplus\{4,5,6\}$. With the convention of $n_r=f(q)$  if $q$ is the $r$-th element of the poset $(P_\alpha,<_{P_{\alpha}})$, we obtain
\begin{align*}
\Gamma\left(P_{(0^2,2,0,1,3)}\right)=\sum_{1\leq n_1\leq n_2\leq n_3\leq n_4< n_5\leq n_6< n_7\leq n_8\leq n_9}
x_{n_1}^0x_{n_2}^0x_{n_3}x_{n_4}x_{n_5}^0x_{n_6}x_{n_7}x_{n_8}x_{n_9}.
\end{align*}

\begin{proof}
By Construction \ref{constp-partextend}, as a labeled poset, the underlying set of $P_\alpha$ is $\{1,2,\cdots,a_k\}$, where $a_k:=i_1+s_1+\cdots+i_k+s_k$.
With the convention of $i_0=s_0=0$ for $p=1,\cdots, k$, we have
\begin{align*}
C_p=\left\{\left. t\in \NN\,\right|\,i_1+\cdots+i_{p-1}+1\leq t\leq i_1+\cdots+i_{p-1}+i_p\right\}
\end{align*}
and
\begin{align*}
P_p=\left\{t\in \NN\,|\,i_1+\cdots+i_k+s_1+\cdots+s_{k-p}+1\leq t\leq i_1+\cdots+i_k+s_1+\cdots+s_{k-p+1}\right\}.
\end{align*}
For a  $P$-partition $f\in A(P_\alpha)$,  we write $n_{r}=f(q)$ if $q$ is the $r$-th element of the poset $(P_\alpha,<_{P_{\alpha}})$.
Then the desired statement follows from the definition of P-partition.
\end{proof}

\subsection{\lwc fundamental quasi-symmetric functions}

As in the case of quasi-symmetric functions, we will next show, in Theorem \ref{fundamentalbasmonomialbs}, that $\Gamma(P_\alpha)$ give a class of \eqsyms, called \lwc fundamental quasi-symmetric functions.
If $\alpha$ is a composition, then we recover the fundamental quasi-symmetric functions.

We first need a partial order on the set of left weak compositions of $n$.
Let $\alpha,\beta$ be two compositions of $n$.
Recall that $\alpha\preceq \beta$, if we can obtain the parts of $\beta$ by adding some adjacent parts of $\alpha$.
Let $\alpha,\beta$ be left weak compositions of $n$, and let $\beta=(0^{j_1},\beta_1,\cdots,0^{j_k},\beta_k)$ where
$j_p$ are nonnegative integers and $\beta_p$ are
positive integers, $p=1, \cdots,k$.
We extend the order $\preceq$ on compositions to the set of left weak compositions of $n$, still denoted by $\preceq$,
by setting $\alpha\preceq\beta$
if we can write
$\alpha=(0^{i_1},\alpha_{1},\cdots,0^{i_k},\alpha_{k})$
such that $0\leq i_p\leq j_p$
and $\alpha_p$ is a composition of $\beta_p$, $p=1, \cdots,k$.
For example, let $\alpha=(1,2,0^2,1,2,1,2)$, $\beta=(3,0^4,1,0^3,3,2)$ be two left weak compositions.
Then we can write $\beta$ as $(3,0^4,1,0^3,3,0^0,2)$
and $\alpha$ as $(1,2,0^2,1,0^0,2,1,0^0,2)$. Clearly,
$$(1,2,0^2,1,0^0,2,1,0^0,2)\preceq (3,0^4,1,0^3,3,0^0,2)$$
so that $\alpha\preceq \beta$.
But $(1,2,0^2,1,0^4,2,1,2)$ is incomparable with $\beta$.

Let  $\alpha=(0^{i_1},s_1,\cdots,0^{i_k},s_k)$ be a left weak composition, where $i_p\geq0$, $s_p>0$,
$p=1,\cdots,k$.
Recall that the \lwc monomial quasi-symmetric function $M_\alpha$ is given by
$$ {M}_{\alpha}=\sum_{1\leq n_1<\cdots<n_{i_1+\cdots+i_{k}+k}}
x_{n_1}^0\cdots x_{n_{i_1}}^0 x_{n_{i_1+1}}^{s_1} \cdots  x_{n_{i_1+\cdots+i_{k-1}+k}}^0
\cdots x_{n_{i_1+\cdots+i_{k}+k-1}}^0 x_{n_{i_1+\cdots+i_{k}+k}}^{s_k}.
$$
Denote
$a_p=i_1+s_1+\cdots+i_p+s_p$, $p=1,2,\cdots,k$.
Then the {\bf left weak composition (\lwc) fundamental  quasi-symmetric function} ${F_\alpha}$ is a formal power series
defined by
\begin{align}\label{fundmbasis1}
 {F}_{\alpha}=\sum x_{n_1}^0\cdots x_{n_{i_1}}^0 x_{n_{i_1+1}}\cdots x_{n_{a_1}}
\cdots x_{n_{a_{k-1}+1}}^0\cdots x_{n_{a_{k-1}+i_{k}}}^0 x_{n_{a_{k-1}+i_{k}+1}}\cdots x_{n_{a_k}},
\end{align}
where the summation is subject to the condition
\begin{align}\label{ni<leqni+1}
1\leq n_1\leq \cdots \leq n_{a_1}<n_{a_1+1}\leq\cdots< n_{a_{k-1}+1}\leq\cdots \leq n_{a_{k}}.
\end{align}

\begin{theorem}\label{fundamentalbasmonomialbs}
Let a left weak composition be written as $\alpha=(0^{i_1},s_1,\cdots,0^{i_k},s_k)$,
where $i_1,\cdots,i_{k}$ are nonnegative integers and $s_1,\cdots,s_k$ are positive integers. Then
\begin{align}\label{fmalphas}
 {F}_{\alpha}= \sum_{\beta\preceq\alpha}c_{\alpha,\beta}  {M}_\beta,
\end{align}
where $c_{\alpha,\beta}=\binom{i_1}{j_1}\cdots
\binom{i_k}{j_k} $  if
$\beta=(0^{j_1},\beta_1,\cdots,0^{j_k},\beta_k)$ with $\beta_p$ a composition of $s_p$ and $0\leq j_p\leq i_p$ for $1\leq p\leq k$.
In particular, $c_{\alpha,\alpha}=1$.
Moreover, when $\alpha$ runs through all left weak compositions, the elements ${F}_{\alpha}$, together with $F_\emptyset=1$, form a basis for $\EQSYM$.
\end{theorem}

As an example of Eq.~(\ref{fmalphas}), we have
\begin{align*}
 {F}_{(0^2,2)}&= {M}_{(0^2,2)}+2 {M}_{(0,2)}+ {M}_{2}+ {M}_{(0^2,1,1)}+2 {M}_{(0,1,1)}+ {M}_{(1,1)}.
\end{align*}

\begin{proof}
For the given left weak composition $\alpha$, denote $a_p:=i_1+s_1+\cdots+i_p+s_p$, $p=1,2,\cdots,k$.
Define
\begin{align*}
\widetilde{ {F}}_{\alpha}=\sum x_{n_1}^0\cdots x_{n_{i_1}}^0 x_{n_{i_1+1}}\cdots x_{n_{a_1}}
\cdots x_{n_{a_{k-1}+1}}^0\cdots x_{n_{a_{k-1}+i_{k}}}^0 x_{n_{a_{k-1}+i_{k}+1}}\cdots x_{n_{a_k}},
\end{align*}
where the summation is subject to the condition in Eq.~\eqref{ni<leqni+1} and
\begin{align*}
n_{a_p+1}<\cdots<n_{a_p+i_{p+1}}<n_{a_p+i_{p+1}+1},\qquad p=0,1,\cdots, k-1,
\end{align*}
with the convention that $a_0=0$. In Eq.~\eqref{fundmbasis1}, we can choose $j_{p+1}$ inequalities in the condition
$n_{a_p+1}\leq n_{a_p+2}\leq\cdots\leq n_{a_p+i_{p+1}}\leq n_{a_p+i_{p+1}+1}$,  where $0\leq j_{p+1}\leq i_{p+1}$, $0\leq p\leq k-1$.
Hence we have
\begin{align*}
 {F}_{\alpha}
=&\sum_{j_1,\cdots,j_{k}}\binom{i_1}{j_1}\cdots
\binom{i_k}{j_k} \widetilde{ {F}}_{(0^{j_1},s_1,\cdots,0^{j_k},s_k)},
\end{align*}
where the summation is subject to the condition $0\leq j_p\leq i_p$, $p=1,\cdots,k$.
Then, by applying the relation of monomial quasi-symmetric functions $M_K$ and fundamental quasi-symmetric functions $F_L$, that is,
$F_L=\sum_{K\preceq L}M_K$, we obtain
\begin{align*}
 {F}_{\alpha}=
\sum_{j_1,\cdots,j_{k}}\binom{i_1}{j_1}\cdots
\binom{i_k}{j_k} \sum_{\alpha_{11},\cdots,\alpha_{1r_1},\cdots,\alpha_{k1},\cdots,
\alpha_{kr_k}}  {M}_{(0^{j_1},\alpha_{11},\cdots,\alpha_{1r_1},\cdots,0^{j_k},\alpha_{k1},\cdots,
\alpha_{kr_k})},
\end{align*}
where $(\alpha_{p1},\cdots,\alpha_{pr_p})$ runs through all compositions of $s_p$, $p=1,\cdots,k$, so that $(\alpha_{p1},\cdots,\alpha_{pr_p})\preceq s_p$.
Hence
$$(0^{j_1},\alpha_{11},\cdots,\alpha_{1r_1},\cdots,0^{j_k},\alpha_{k1},\cdots,a_{kr_k})\preceq\alpha.$$
Thus, Eq.~\eqref{fmalphas} follows and $c_{\alpha,\alpha}=1$ is clearly true.

Since $\EQSYM=\bigoplus\limits_{n\geq 0}\EQSYM_n$, to prove the last statement, we only need to show that, for each $n\geq 1$,
\begin{align*}
\{ {F}_{\alpha}|\alpha\ {\rm is\ a\ left\ weak\ composition\ of}\ n\}
\end{align*}
is a basis for $\EQSYM_n$.

By the definition of the partial order $\preceq$, for any given left weak composition $\alpha$ of size $n$, there are only finite many left weak compositions less than $\alpha$. These left weak compositions also have size $n$.
It follows from Eq.~\eqref{fmalphas} that
the transition matrix which expresses the $ {F}_{\alpha}$ in terms of the $ {M}_{\beta}$,
with respect to any linear order of left weak compositions that extends $\preceq$, is upper triangular with $1$ on the main diagonal. So it is an invertible matrix, which shows that
$ {F}_{\alpha}$ is a basis for $\EQSYM_n$ knowing that $M_\alpha$ for $\alpha$ of size $n$ is a basis of $\EQSYM_n$.
\end{proof}

Indeed, as in the case of quasi-symmetric functions, we can also express $ {M}_\alpha$ in terms of the basis $ {F}_\beta$.

\begin{theorem}\label{fundamentalbastomonomialbs}
 For any left weak composition $\alpha$, we have
\begin{align}\label{fphasmaleq}
 {M}_\alpha=\sum_{\beta\preceq\alpha}(-1)^{\ell(\beta)-\ell(\alpha)}c_{\alpha,\beta}  {F}_\beta.
\end{align}
\end{theorem}
For example,
\begin{align*}
M_{(0^2,2)}&= {F}_{(0^2,2)}-2 {F}_{(0,2)}+ {F}_{2}- {F}_{(0^2,1,1)}+2 {F}_{(0,1,1)}- {F}_{(1,1)}.
\end{align*}

\begin{proof}
By the definition of the order $\preceq$, there are only finite left weak compositions less than $\alpha$, and
$1^n$ is the least element of the set of left weak compositions with size $n$.
Clearly, $ {M}_{1^n}= {F}_{1^n}$.
We next prove that Eq.~\eqref{fphasmaleq} holds for a left weak composition $\alpha$ provide it holds for
all left weak compositions less than $\alpha$.

From Eq.~\eqref{fmalphas} we obtain
$ {M}_\alpha= {F}_\alpha-\sum\limits_{\gamma\prec\alpha}c_{\alpha,\gamma}  {M}_\gamma$,
which together with the induction hypothesis yields
\begin{align}\label{malpahseqbga}
 {M}_\alpha=& {F}_\alpha-\sum_{\gamma\prec\alpha}c_{\alpha,\gamma}
\sum_{\beta\preceq\gamma}(-1)^{\ell(\beta)-\ell(\gamma)}c_{\gamma,\beta}  {F}_\beta\cr
=& {F}_\alpha-\sum_{\beta\preceq\gamma\preceq\alpha}(-1)^{\ell(\beta)-\ell(\gamma)}c_{\alpha,\gamma}
c_{\gamma,\beta}  {F}_\beta
+\sum_{\beta\preceq\gamma=\alpha}(-1)^{\ell(\beta)-\ell(\gamma)}c_{\alpha,\gamma}
c_{\gamma,\beta}  {F}_\beta\cr
=& {F}_\alpha-\sum_{\beta\preceq\alpha}\left(\sum_{\beta\preceq\gamma\preceq\alpha}
(-1)^{\ell(\beta)-\ell(\gamma)}c_{\alpha,\gamma}
c_{\gamma,\beta} \right) {F}_\beta
+\sum_{\beta\preceq\alpha}(-1)^{\ell(\beta)-\ell(\alpha)}
c_{\alpha,\beta}  {F}_\beta,
\end{align}
since $\beta\preceq\gamma\preceq\alpha$ can be divided into two disjoint cases $\beta\preceq\gamma\prec\alpha$ and $\beta\preceq\gamma=\alpha$.

Let $\alpha=(0^{i_1},s_1,\cdots,0^{i_k},s_k)$ be a left weak composition of $n$ where $s_p>0$, $1\leq p\leq k$.
Assume that $\beta=(0^{t_1},\beta_1,\cdots,0^{t_k},\beta_k)$,
$\gamma=(0^{j_1},\gamma_1,\cdots,0^{j_k},\gamma_k)$ with $\beta\preceq\gamma\preceq\alpha$.
Then for each $p$ with $1\leq p\leq k$, $\beta_p$ and $\gamma_p$ are compositions of $s_p$ with $\beta_p\preceq\gamma_p$.
Denote $\overline{\alpha}=(s_1,\cdots,s_k)$ and analogously $\overline{\beta}=({\beta_1},\cdots,{\beta_k})$
and $\overline{\gamma}=(\gamma_1,\cdots,\gamma_k)$.
Then $\ell(\alpha)=\ell(\overline{\alpha})+i_1+\cdots+i_{k}$.
If $\ell(\alpha)=\ell(\overline{\alpha})$, then $\alpha$ is a composition so that $ {M}_\alpha$ is a monomial  quasi-symmetric
function and $ {F}_\alpha$ is a fundamental quasi-symmetric functions.  So $c_{\alpha,\beta}=1$ and hence Eq.~\eqref{fphasmaleq} reduces to
$ {M}_\alpha=\sum\limits_{\beta\preceq\alpha}(-1)^{\ell(\beta)-\ell(\alpha)} {F}_\beta$, which
holds as a well-known fact for quasi-symmetric functions.
Next we suppose that at least one of $i_p$ is positive, $1\leq p\leq k$.
Thus, for fixed $\alpha$ and $\beta$, we have
\begin{align*}
\sum_{\beta\preceq\gamma\preceq\alpha}(-1)^{\ell(\gamma)}
c_{\alpha,\gamma}c_{\gamma,\beta}
=&\sum_{\beta\preceq\gamma\preceq\alpha}(-1)^{\ell(\overline{\gamma})+j_1+\cdots+j_{k}}
\binom{i_1}{j_1}\cdots
\binom{i_k}{j_k} \binom{j_1}{t_1}\cdots
\binom{j_k}{t_k}.
\end{align*}
Notice that $\beta\preceq\gamma\preceq\alpha$ is equivalent to $\overline{\beta}\preceq\overline{\gamma}\preceq\overline{\alpha}$
and $t_p\leq j_p\leq i_p$ for all $1\leq p\leq k$, so
\begin{align}\label{coefficcalphabetagamms}
\sum_{\beta\preceq\gamma\preceq\alpha}(-1)^{\ell(\gamma)}
c_{\alpha,\gamma}c_{\gamma,\beta}
=&\sum_{ \overline{\beta}\preceq\overline{\gamma}\preceq\overline{\alpha}}
(-1)^{\ell(\overline{\gamma})}\times
\prod_{p=1}^k\left(\sum_{j_p=t_p}^{i_p}(-1)^{j_p}
\binom{i_p}{j_p}\binom{j_p}{t_p}\right).
\end{align}
Applying the identities
\begin{align*}
\binom{m}{n}\binom{n}{q}=\binom{m}{q}\binom{m-q}{n-q} \qquad {\rm and}
\qquad \sum_{n=0}^{m}(-1)^{n}\binom{m}{n}
=\begin{cases}
0,&m> 0,\\
1,&m=0,
\end{cases}
\end{align*}
we obtain
\begin{align*}
\sum_{j_p=t_p}^{i_p}(-1)^{j_p}\binom{i_p}{j_p}\binom{j_p}{t_p}
=(-1)^{t_p}\binom{i_p}{t_p}\sum_{j_p=t_p}^{i_p}(-1)^{j_p-t_p}\binom{i_p-t_p}{j_p-t_p}=
\begin{cases}
0,& i_p>t_p,\\
(-1)^{i_p},& i_p=t_p,
\end{cases}
(1\leq p\leq k),
\end{align*}
which yields $\sum\limits_{\beta\preceq\gamma\preceq\alpha}(-1)^{\ell(\gamma)}
c_{\alpha,\gamma}c_{\gamma,\beta}=0$ if there exists $p$ with $1\leq p\leq k$ such that $i_p>t_p$. Otherwise, we have
$i_p=t_p$ for all $1\leq p\leq k$ so that Eq.~\eqref{coefficcalphabetagamms} becomes
\begin{align}\label{coeeqcases21}
\sum_{\beta\preceq\gamma\preceq\alpha}(-1)^{\ell(\gamma)}
c_{\alpha,\gamma}c_{\gamma,\beta}
=&\begin{cases}
0,& {\rm there\ exists}\ 1\leq p\leq k\ {\rm such\ that}\ i_p>t_p,\\
\sum\limits_{\overline{\beta}\preceq\overline{\gamma}\preceq\overline{\alpha}}(-1)^{\ell(\overline{\gamma})}\times
\prod_{p=1}^{k}(-1)^{i_p}, &i_p=t_p\ {\rm for\ all}\ 1\leq p\leq k;
\end{cases}\cr
=&\begin{cases}
0,& {\rm there\ exists}\ 1\leq p\leq k\ {\rm such\ that}\ i_p>t_p,\\
(-1)^{i_1+\cdots+i_{k}}\sum\limits_{\overline{\beta}\preceq\overline{\gamma}\preceq\overline{\alpha}}(-1)^{\ell(\overline{\gamma})},
&i_p=t_p\ {\rm for\ all}\ 1\leq p\leq k.
\end{cases}
\end{align}
Noting that $s_p$ are positive integers, $\overline{\beta}\preceq\overline{\gamma}\preceq\overline{\alpha}$
means that $\beta_p$ and $\gamma_p$ are compositions of $s_p$ with $\beta_p\preceq\gamma_p\preceq s_p$, $1\leq p\leq k$.
So we can write
\begin{align}\label{coeeqcases22}
\sum_{\overline{\beta}\preceq\overline{\gamma}\preceq\overline{\alpha}}(-1)^{\ell(\overline{\gamma})}
=\sum_{\overline{\beta}\preceq\overline{\gamma}\preceq\overline{\alpha}}
(-1)^{\ell({\gamma_1})+\cdots+\ell({\gamma_k})}
=\prod_{p=1}^{k}\left(\sum_{{\beta_p}\preceq {\gamma_p}\preceq {s_p}}(-1)^{\ell(\gamma_p)}\right).
\end{align}

Recall that the compositions of $n$ are in one-to-one correspondence with the subsets of $\{1,2,\cdots$,
$n-1\}$,
and the bijection is given by sending a composition $I$ to $\set(I)$.
So $\ell(I)=\sharp \set(I)+1$ and $\set(I)\supseteq \set(J)$ if and only if $I\preceq\ J$, where $\sharp \set(I)$ denotes the cardinality of the set $\set(I)$.
In particular,  since $s_p$ is a positive integer,
we have $\set(s_p)=\emptyset$. Then for fixed $\beta_p$ and $s_p$, we obtain
\begin{align*}
\sum_{\beta_p\preceq\gamma_p\preceq s_p}(-1)^{\ell(\gamma_p)}
=\sum_{\set(\beta_p)\supseteq \set(\gamma_p)\supseteq \emptyset}(-1)^{\sharp \set(\gamma_p)+1}
=\begin{cases}
0,& \beta_p\prec s_p,\\
-1, & \beta_p=s_p.
\end{cases}
\end{align*}
Then, by Eqs.~\eqref{coeeqcases21} and \eqref{coeeqcases22}, we see that for fixed $\alpha$ and $\beta$ with $\beta\preceq \alpha$,
\begin{align*}
\sum_{\beta\preceq\gamma\preceq\alpha}(-1)^{\ell(\gamma)}
c_{\alpha,\gamma}c_{\gamma,\beta}
=\begin{cases}
0,& \beta\prec\alpha,\\
(-1)^{k+i_1+\cdots+i_{k}},& \beta=\alpha;
\end{cases}
=\begin{cases}
0,& \beta\prec\alpha,\\
(-1)^{\ell(\alpha)},& \beta=\alpha.
\end{cases}
\end{align*}
Hence,
\begin{align*}
\sum_{\beta\preceq\alpha}\left(\sum_{\beta\preceq\gamma\preceq\alpha}
(-1)^{\ell(\beta)-\ell(\gamma)}c_{\alpha,\gamma}
c_{\gamma,\beta} \right) {F}_\beta
=&\sum_{\beta\preceq\alpha}(-1)^{\ell(\beta)} \left(\sum_{\beta\preceq\gamma\preceq\alpha}
(-1)^{\ell(\gamma)}c_{\alpha,\gamma}
c_{\gamma,\beta} \right) {F}_\beta= {F}_\alpha.
\end{align*}

Note that $c_{\alpha,\alpha}=1$, so Eq.~\eqref{malpahseqbga} eventually reduces to
\begin{align*}
 {M}_\alpha=&\sum_{\beta\preceq\alpha}(-1)^{\ell(\beta)-\ell(\alpha)}
c_{\alpha,\beta}  {F}_\beta,
\end{align*}
giving the desired result. This completes the proof.
\end{proof} \vspace{3mm}

\section{Application to multiple zeta values and their $q$-analogs}\label{sec:application}

In this section, we extend the connection of quasi-symmetric functions with MZVs recalled in  Eq.~\eqref{eq:mzv} to a connection of \eqsyms with a class of special values in multiple variables, called the left weak composition (\lwc) MZVs, and with $q$-MZVs.

\noindent
{\bf Convention.} Throughout this section, we adapt the convention that $\binom{p}{q}=0$ if $p<q$ or $q<0$.

\subsection{Left weak composition multiple zeta values}\label{sec:extendMZV}
We introduce the concept of a left weak composition MZV and apply \eqsyms to establish a
formula for left weak composition MZVs.

Let  $\alpha=(0^{i_1},s_1,\cdots,0^{i_k},s_k)$ be a left weak composition, where $i_p\geq0$, $s_p>0$ and
$p=1,\cdots,k$.
Recall from Lemma \ref{lemmalphamon} that the left weak monomial quasi-symmetric function $M_\alpha$ is given by
\begin{equation*}
 {M}_{(0^{i_1},s_1,0^{i_2},s_2,\cdots,0^{i_k},s_k)}
=\sum_{1\leq n_1<n_2<\cdots<n_k}\binom{n_1-1}{i_1}\binom{n_2-n_1-1}{i_2}\cdots\binom{n_k-n_{k-1}-1}{i_k}x_{n_1}^{s_1}x_{n_2}^{s_2}\cdots
x_{n_k}^{s_k}.
\end{equation*}
In analogous to the relationship between quasi-symmetric functions and MZVs recalled in Eq.~(\ref{eq:mzv}), this leads to the following left weak MZVs.

\begin{defn}
Let $\alpha=(0^{i_1},s_1,\cdots,0^{i_k},s_k)$ be a left weak composition and denote $I=(i_1,\cdots,i_k)$.
Define the {\bf \emzv} (\lwcmzv for short)
\begin{equation}
\zeta(\alpha)=\zeta(s_1,\cdots,s_k;I):=\sum_{1\leq n_1<n_2<\cdots<n_k}\binom{n_1-1}{i_1} \binom{n_2-n_1-1}{i_2}\cdots\binom{n_k-n_{k-1}-1}{i_k} \frac{1}{n_1^{s_1}n_2^{s_2} \cdots n_k^{s_k}}.
\mlabel{eq:emzv}
\end{equation}
\end{defn}
It converges for integers $s_p\geq i_p+2, 1\leq p\leq k.$
When $i_p$ is zero for each $p$ with $1\leq p\leq k$, we recover the usual MZVs.

By Theorem \ref{thmmonproduct}, \lwcmzvs satisfy the quasi-shuffle rule, that is, $\zeta(\alpha)\zeta(\beta)=\zeta(\alpha*\beta)$ if $\alpha,\beta$ are
left weak compositions.
Therefore we obtain, for example,
$$\zeta(a;1)\zeta(a;1)=2\zeta(a,a;1,1)+4\zeta(a,a;2,0) +4\zeta(a,a;1,0) +2\zeta(2a,2)+\zeta(2a;1).$$
We need some preparation before proving a general formula.

\begin{lemma}\label{00bstuff}
Let $b, m,n$ be nonnegative integers. Then
\begin{align}\label{eq00bst6uff}
0^m*(0^n,b)=\sum_{i=0}^{m}\sum_{k=n}^{m+n-i}\binom{k}{n}
\binom{n+1}{i+k-m+1}(0^k,b,0^i).
\end{align}
\end{lemma}
\begin{proof}
The case when at least one of $m,n$ is zero is trivial. Next we assume that $m,n\geq1$.
We use the induction on $m+n$ with $m,n\geq1$. The case of $m+n=2$, that is, $m=n=1$, can be verified directly.
Now suppose that Eq.~\eqref{eq00bst6uff} holds for $m+n< j$ and consider the case when $m+n=j$.
By the definition of the quasi-shuffle product and applying the induction hypothesis, we have
\begin{align*}
&0^m*(0^n,b)\\
=&(0,0^{m-1}*0^nb)+(0,0^m*0^{n-1}b)+(0,0^{m-1}*0^{n-1}b)\cr
=&\sum_{i=0}^{m-1}\sum_{j=n}^{m+n-i-1}\binom{j}{n}
\binom{n+1}{i+j-m+2}(0^{j+1},b,0^i)+\sum_{i=0}^{m}\sum_{j=n-1}^{m+n-i-1}\binom{j}{n-1}
\binom{n}{i+j-m+1}(0^{j+1},b,0^i)\cr
&+\sum_{i=0}^{m-1}\sum_{j=n-1}^{m+n-i-2}\binom{j}{n-1}
\binom{n}{i+j-m+2}(0^{j+1},b,0^i)\cr
=&
\sum_{i=0}^{m}\sum_{k=n}^{m+n-i}\left(
\binom{k-1}{n}
\binom{n+1}{i+k-m+1}+\binom{k-1}{n-1}
\binom{n}{i+k-m}+\binom{k-1}{n-1}
\binom{n}{i+k-m+1}\right)
(0^k,b,0^i),
\end{align*}
where for the last equation, we take $j=k-1$ and use $\binom{p}{q}=0$ if $p<q$.
By Pascal's identity, one has
\begin{align*}
\binom{k-1}{n}
\binom{n+1}{i+k-m+1}+\binom{k-1}{n-1}
\binom{n}{i+k-m}+\binom{k-1}{n-1}
\binom{n}{i+k-m+1}
=
\binom{k}{n}
\binom{n+1}{i+k-m+1}.
\end{align*}
Therefore, we have
\begin{align*}
0^m*(0^n,b)=\sum_{i=0}^{m}\sum_{k=n}^{m+n-i}\binom{k}{n}
\binom{n+1}{i+k-m+1}(0^k,b,0^i),
\end{align*}
completing the induction.
\end{proof}

If $b=0$, we could obtain a formula for $0^m*0^n$.
\begin{coro}\label{oomn}
For any nonnegative integers $m,n$,
\begin{align*}
0^m*0^n=\sum_{k=n}^{m+n}\binom{k}{n}\binom{n}{k-m}0^{k}.
\end{align*}
\end{coro}
\begin{proof}
It follows from Lemma \ref{00bstuff} that
\begin{align*}
0^m*0^{n}=&0^m*(0^{n-1},0)=\sum_{i=0}^{m}\sum_{j=n-1}^{m+n-i-1}\binom{j}{n-1}
\binom{n}{i+j-m+1}0^{j+i+1}.
\end{align*}
Taking $k=j+i+1$, then
\begin{align*}
0^m*0^{n}=\sum_{i=0}^{m}\sum_{k=n+i}^{m+n}\binom{k-i-1}{n-1}
\binom{n}{k-m}0^{k}.
\end{align*}
Exchanging the order of $i$ and $k$ in the summation, we obtain
\begin{align*}
0^m*0^{n}=\sum_{k=n}^{m+n}\sum_{i=0}^{k-n}\binom{k-i-1}{n-1}
\binom{n}{k-m}0^{k}.
\end{align*}
It follows from Pascal's identity that
\begin{align*}
\sum_{i=0}^{k-n}\binom{k-i-1}{n-1}
=&\sum_{i=0}^{k-n-1}\binom{k-i-1}{n-1}+\binom{n}{n}=\binom{k}{n},
\end{align*}
so that
\begin{align*}
0^m*0^{n}
=\sum_{k=n}^{m+n}\binom{k}{n}\binom{n}{k-m}0^{k},
\end{align*}
as required.
\end{proof} \vspace{3mm}

Combining Lemma \ref{00bstuff} and Corollary \ref{oomn}, we obtain an explicit formula for $(0^m,a)*(0^n,b)$.
\begin{coro}\label{omastuffle0mb}
Let $a,b,m,n$ be nonnegative integers. Then
\begin{align*}
(0^m,a)*(0^n,b)
&=\sum_{i=0}^{m}\sum_{k=n}^{m+n-i}\binom{k}{n}
\binom{n+1}{i+k-m+1}(0^k,b,0^i,a)\cr
&+\sum_{i=0}^{n}\sum_{k=m}^{m+n-i}\binom{k}{m}
\binom{m+1}{i+k-n+1}(0^k,a,0^i,b) +\sum_{k=n}^{m+n}\binom{k}{n}\binom{n}{k-m}(0^{k},a+b).
\end{align*}
\end{coro}
\begin{proof}
It suffices to note that the quasi-shuffle product is commutative and
$(0^m,a)*(0^n,b)=(0^m*0^nb,a)+(0^ma*0^n,b)+(0^m*0^n,a+b)$. So the desired result follows from Lemma \ref{00bstuff} and Corollary \ref{oomn} directly.
\end{proof}

Now we can prove a quasi-shuffle formula for \lwcmzvs.
\begin{theorem}\label{binomazetafomucx}
Let $a,b,m,n$ be nonnegative integers with $a\geq m+2$, $b\geq n+2$. Then
\begin{align*}
\zeta(a;m)\zeta(b;n)=&\sum_{i=0}^{m}\sum_{k=n}^{m+n-i}\binom{k}{n}
\binom{n+1}{i+k-m+1}\zeta(b,a;k,i)\cr
&+\sum_{i=0}^{n}\sum_{k=m}^{m+n-i}\binom{k}{m}
\binom{m+1}{i+k-n+1}\zeta(a,b;k,i)+\sum_{k=n}^{m+n}\binom{k}{n}\binom{n}{k-m}\zeta(a+b; {k}).
\end{align*}
\end{theorem}

\begin{proof}
Since \lwcmzvs are obtained from \lwc monomial quasi-symmetric functions by evaluation $x_i\mapsto 1/i$, the multiplication of these \lwc quasi-symmetric functions satisfy the quasi-shuffle product. Thus the theorem follows from Corollary~\ref{omastuffle0mb}.
\end{proof}

We also obtain a relation between \lwcmzvs and MZVs.
Recall that  $s(i,k)$ is the Stirling number of the first kind defined by
the equation
\begin{align}\label{stirlingfirstkd1}
(t)_i:=t(t-1)\cdots (t-i+1)=\sum_{k=0}^i s(i,k) t^k.
\end{align}

\begin{theorem}
Let $a,b,m,n$ be nonnegative integers  with $a\geq m+2$ and $b\geq n+2$. Then
\begin{align*}
\zeta&(a;m)\zeta(b;n)=\sum_{k=0}^m \sum_{l=0}^n \frac{s(m,k)}{m!} \frac{s(n,l)}{n!} \zeta(a-k) \zeta(b-l)
-\sum_{k=0}^m \sum_{l=0}^n\frac{s(m,k)}{m!} \frac{s(n,l)}{(n-1)!}\zeta(a-k)\zeta(b-l+1)\cr
&-\sum_{k=0}^m\sum_{l=0}^n \frac{s(m,k)}{(m-1)!}\frac{s(n,l)}{n!}\zeta(a-k+1)\zeta(b-l)
+\sum_{k=0}^m \sum_{l=0}^n\frac{s(m,k)}{(m-1)!}\frac{s(n,l)}{(n-1)!}\zeta(a-k+1)\zeta(b-l+1).
\end{align*}
\end{theorem}
\begin{proof}
From Eq.~\eqref{stirlingfirstkd1} it follows that for any nonnegative integers $i$ and $n$, we have
\begin{align*}
\binom{n-1}{i}&=\frac{n(n-1)\cdots (n-i+1)}{i!}\frac{n-i}{n}= \sum_{k=0}^i\frac{s(i,k)}{i!}(n^{k}-in^{k-1}).
\end{align*}
Plugging this into $\zeta(a;m)$, we obtain
$$
\zeta(a;m)= \sum_{n\geq 1}\binom{n-1}{m}\frac{1}{n^a}
= \sum_{k=0}^m \frac{s(m,k)}{m!} \zeta(a-k) - \sum_{k=0}^m \frac{s(m,k)}{(m-1)!}\zeta(a-k+1)
$$
and similarly for $\zeta(b;n)$, which gives the desired identity.
\end{proof}

As illustrated in the above proof, an \lwcmzv can be expressed as a linear combination of MZVs with rational coefficients. But the combination can be quite complicated. So it is useful to study \lwcmzvs directly.

\subsection{A $q$-analog of multiple zeta values}\label{sec:aqmzv}
In \cite{Zhao2014a}, the third author studied the double shuffle relations
and the duality relations among various $q$-analogs of the MZVs.
One candidate, first studied by Ohno et al.\ \cite{OOZ2012}, is defined as follows\footnote{The order of indices in the definition \eqref{equ:qMZVtypeII} is opposite to the one used in \cite{Zhao2014a}.}.
Let $q$ be a fixed complex number such that $|q|<1$. For any positive integer $n$
we set $[n]_q=1+q+\dots+q^{n-1}=(1-q^n)/(1-q)$. For
$k$ complex variables $s_1,\dots, s_k$ we define the
$q$-analogs of the multiple zeta function of depth $k$ by
\begin{equation}\mlabel{equ:qMZVtypeII}
\zeta_q(s_1,\dots,s_k)
:=\sum_{1\leq n_1<\dots<n_k} \frac{q^{n_1s_1+\cdots+n_ks_k}}{[n_1]_q^{s_1}\cdots[n_k]_q^{s_k}}.
\end{equation}
By \cite[Prop.~2.2]{Zhao2007c} this function converges if ${\rm Re}(s_j+\cdots+s_k)>0$ for
all $j=1,\dots,k$. In particular, the special value $\zeta_q(\alpha)$ is well-defined
for all $\alpha\in \NN^{k-1}\times\PP$. We call it a {\bf{ $q$-analog of multiple zeta value}}
($q$-MZV for short).

Similar to the lwcmzvs, by evaluating $x_n$ at $q^n/[n]_q$,
we can send $M_\alpha$ to the $q$-MZV $\zeta_q(\alpha)$. Again,
the quasi-shuffle relation or stuffle relation of the  $q$-MZVs is a
direct consequence of the quasi-shuffle product on $\EQSYM$. Hence
the following result can be proved in the same way as that of Theorem~\ref{binomazetafomucx}.

For $I=(i_1,\cdots,i_k)\in \NN^k$ and $(s_1,\cdots,s_k)\in \PP^k$,
we write
\begin{equation*}
\zeta_q(s_1,\cdots,s_k;I):=\zeta_q(0^{i_1},s_1,\cdots,0^{i_k},s_k).
\end{equation*}

\begin{theorem}
Let $a,b,m,n$ be nonnegative integers with $a\geq m+2$, $b\geq n+2$. Then
\begin{align*}
\zeta_q(a;m)\zeta_q(b;n)=&\sum_{i=0}^{m}\sum_{k=n}^{m+n-i}\binom{k}{n}
\binom{n+1}{i+k-m+1}\zeta_q(b,a;k,i)\cr
+&\sum_{i=0}^{n}\sum_{k=m}^{m+n-i}\binom{k}{m}
\binom{m+1}{i+k-n+1}\zeta_q(a,b;k,i)
+\sum_{k=n}^{m+n}\binom{k}{n}\binom{n}{k-m}\zeta_q(a+b; {k}).
\end{align*}
\end{theorem}

Unlike the \lwcmzvs, $q$-MZVs with 0 argument appear naturally
due to the following duality relations.
\begin{theorem} \emph{(\cite[Theorem~8.3]{Zhao2014a})}
For $(s_1,\cdots,s_k),(t_1,\cdots,t_k)\in \PP^k$, we have
\begin{align*}
\zeta_q(t_1,\cdots,t_k;s_1-1,\cdots,s_k-1)
=\zeta_q(s_k,\cdots,s_1;t_k-1,\cdots,t_1-1).
\end{align*}
\end{theorem}

Recall that one can equip the  $q$-MZVs with a $q$-shuffle structure by
applying a suitable Rota--Baxter algebra as follows.
Let $A=\{\rho,y\}$ be an alphabet of two letters
and $A^*$  be the set of words generated by $A$.
Define $\fA=\QQ\langle \rho,y\rangle$ to be the noncommutative
polynomial $\QQ$-algebra of $A^*$. Let $\myone$ be the empty word.
Then we can define a $q$-shuffle product $\Cup$ on $\fA$ recursively
by first setting  $\myone\Cup\bfu=\bfu\Cup\myone=\bfu$
for any word $\bfu\in A^*$.
Then for any words $\bfu,\bfv\in A^*$, we define
\begin{align}
 (y\bfu)\Cup\bfv=&\,\bfu\Cup(y\bfv)=y(\bfu\Cup\bfv),	 \label{shuffle-y}\\
\rho\bfu\Cup\rho\bfv=&\,\rho (\bfu\Cup\rho\bfv)
        +\rho (\rho\bfu\Cup\bfv)+(1-q)\rho (\bfu\Cup\bfv).	 \label{shuffle-rhorho}
\end{align}
We call all the words in $\fA^0:=\rho \fA$ \emph{admissible}.

\begin{prop} \emph{(\cite{Zhao2014a})}
The algebra $(\fA^0,\Cup)$ is associative and commutative.
Moreover, there is an algebra homomorphism
\begin{align*}
     \zeta_q: (\fA^0, \Cup) \longrightarrow &\, \RR[\![q]\!] \cr
     \rho^{s_1}y\cdots \rho^{s_k}y \longmapsto &\, \zeta_q(s_k,\dots,s_1).
\end{align*}
\end{prop}

\begin{lemma}\mlabel{lem:wordFormDZshuffle}
If we set $q=1$ then for any $a,b,m,n\in\PP$, we have
\begin{equation}\label{equ:dzshuffle}
   \rho^a y^m\Cup  \rho^b y^n=
   \sum_{i=0}^{b-1}\binom{a+i-1}{i}\rho^{a+i} y^m \rho^{b-i} y^n
   + \sum_{j=0}^{a-1}\binom{b+j-1}{j}\rho^{b+j} y^n \rho^{a-j} y^m.
\end{equation}
\end{lemma}
\begin{proof}
We proceed by induction on $a+b$. If $a=b=1$ then
by Eq. \eqref{shuffle-y} and \eqref{shuffle-rhorho} we have
\begin{equation*}
     \rho  y^m\Cup  \rho y^n=\rho(y^m\Cup  \rho y^n)+\rho(\rho y^m\Cup y^n)
     =\rho  y^m\rho y^n+\rho  y^n\rho y^m.
\end{equation*}
So the lemma holds. Now if one of $a$ and $b$ is equal to 1,
say, $a=1$ and $b>1$. Then
\begin{equation*}
   \rho y^m\Cup  \rho^b y^n=\rho(y^m\Cup  \rho^b y^n)
   +\rho(\rho y^m\Cup  \rho^{b-1} y^n)
  = \rho y^m\rho^b y^n +
  \sum_{i=0}^{b-2} \rho^{i+2} y^m \rho^{b-i-1} y^n
   + \rho^{b} y^n \rho  y^m
\end{equation*}
by the induction assumption. So the lemma follows from the substitution $i\to i-1$.
In general, assuming $a,b>1$, we have
\begin{align*}
       \rho^a y^m\Cup  \rho^b y^n=&\,\rho(\rho^{a-1} y^m\Cup \rho^b y^n)
       +\rho(\rho^a y^m\Cup \rho^{b-1} y^n) \cr
  =&\, \sum_{i=0}^{b-1}\binom{a+i-2}{i}\rho^{a+i} y^m \rho^{b-i} y^n
   + \sum_{j=0}^{a-2}\binom{b+j-1}{j}\rho^{b+j+1} y^n \rho^{a-j-1} y^m \cr
  +&\,\sum_{i=0}^{b-2}\binom{a+i-1}{i}\rho^{a+i+1} y^m \rho^{b-i-1} y^n
   + \sum_{j=0}^{a-1}\binom{b+j-2}{j}\rho^{b+j} y^n \rho^{a-j} y^m .
\end{align*}
Applying the substitutions $j\to j-1$ in the second sum and $i\to i-1$ in the third sum,
and then using Pascal's identity $\binom{k}{i}+\binom{k}{i-1}=\binom{k+1}{i}$,
we derive the lemma immediately.
\end{proof}

\begin{coro}\mlabel{cor:shuffle}
For any $a,b,m,n\in\NN$ such that $a\geq m+2$ and $b\geq n+2$, we have
\begin{equation}\notag
 \zeta(a;m)\zeta(b;n)=
   \sum_{i=0}^{b-1}\binom{a+i-1}{i}\zeta(b-i,a+i;n,m)
   + \sum_{j=0}^{a-1}\binom{b+j-1}{j}\zeta(a-j,b+j;m,n).
\end{equation}
\end{coro}
\begin{proof}
Replace $m$ (resp.\ $n$) by $m+1$ (resp.\ $n+1$) in Lemma~\ref{lem:wordFormDZshuffle}
and use the fact that $\lim\limits_{q\to 1} \zeta_q(\alpha)=\zeta(\alpha)$
whenever $\zeta(\alpha)$ converges.
\end{proof}

Combining with Theorem~\ref{binomazetafomucx}, this provides a generalization of
the finite double shuffle relations of MZVs to the product of two \lwcmzvs.

\begin{exam}\label{exampleDBSF}
By Theorem \ref{binomazetafomucx}, we have
\begin{equation*}
\zeta(3;1)^2= 4\zeta(3,3;1,0)+4\zeta(3,3;2,0)+2\zeta(3,3;1,1)+\zeta(6;1)+2\zeta(6;2).
\end{equation*}
On the other hand, by Corollary~\ref{cor:shuffle},
\begin{equation*}
\zeta(3;1)^2=2\zeta(3,3;1,1)+6\zeta(2,4;1,1)+12\zeta(1,5;1,1).
\end{equation*}
Thus we obtain a relation among the left weak composition zeta and double zeta values:
\begin{equation*}
4\zeta(3,3;1,0)+4\zeta(3,3;2,0)+\zeta(6;1)+2\zeta(6;2)=6\zeta(2,4;1,1)+12\zeta(1,5;1,1).
\end{equation*}
\end{exam}

Note that $\zeta(a;0)=\zeta(a)$, so
if we take $m,n$ to be $0$, then Corollary \ref{cor:shuffle} implies the
well-known {\bf Euler decomposition formula}
\begin{align}\mlabel{Eudeform}
\zeta(a)\zeta(b)=\sum_{i=0}^{b-1} \binom{a+i-1}{i}\zeta(b-i,a+i)
+\sum_{j=0}^{a-1} \binom{b+j-1}{j}\zeta(a-j,b+j),\quad a, b\geq 2,
\end{align}
which can also be proved by a shuffle product argument (see {\em e.g.} \cite{GX}).
\smallskip

\noindent
{\bf Acknowledgements.}
H. Yu thanks the hospitality and stimulating environment provided by NYU Polytechnic School of Engineering.
This work was supported by National Natural Science Foundation of China $($No. 11426183, 11501467$)$,
Chongqing Research Program of Application Foundation and Advanced Technology $($No. cstc2014jcyjA00028$)$.

\end{document}